\input ./style/arxiv-general.cfg
\documentclass[aop,MSNbibl,dvips]{arximspdf}
\makeatletter
   \@ifpackageloaded{graphicx}{}{\usepackage{graphicx}}
\makeatother

\doi{10.1214/14-AOP977}
\volume{44}
\issue{1}
\pubyear{2016}
\firstpage{399}
\lastpage{443}
\docsubty{FLA}

\makeatletter
\newcommand{\ess}{\operatorname{ess}}
\newcommand{\ds}{\displaystyle}
\newcommand{\rrvert}{\vert}
\newcommand{\rrVert}{\Vert}
\newcommand{\llvert}{\vert}
\newcommand{\llVert}{\Vert}
\newproclaim{notation}{Notation}
\newtheorem{theorem}{Theorem}[section]
\newproclaim{condition}[theorem]{Condition}
\newtheorem{corollary}[theorem]{Corollary}
\newproclaim{definition}[theorem]{Definition}
\newproclaim{hypothesis}[theorem]{Hypothesis}
\newtheorem{lemma}[theorem]{Lemma}
\newtheorem{proposition}[theorem]{Proposition}
\newproclaim{remark}[theorem]{Remark}
\makeatother

\begin{document}
\begin{frontmatter}

\title{Gaussian-type lower bounds for the density of solutions of SDEs
driven by fractional Brownian~motions}
\runtitle{Gaussian-type lower bounds for density of fractional SDEs\hspace*{10pt}}

\begin{aug}
\author[A]{\fnms{M.}~\snm{Besal\'u}\corref{}\ead[label=e1]{mireiabesalu@gmail.com}},
\author[B]{\fnms{A.}~\snm{Kohatsu-Higa}\thanksref{t1}\ead[label=e2]{arturokohatsu@gmail.com}}
\and
\author[C]{\fnms{S.}~\snm{Tindel}\ead[label=e3]{samy.tindel@univ-lorraine.fr}}
\runauthor{M. Besal\'{u}, A. Kohatsu-Higa and S. Tindel}
\affiliation{Universit\'e de Lorraine, Ritsumeikan University and Japan Science and Technology Agency, and Universit\'e de Lorraine}
\address[A]{M. Besal\'{u}\\
INRIA Nancy Grand Est\\
\quad and Institut {\'E}lie Cartan Lorraine\\
Universit\'e de Lorraine\\
B.P. 239\\
54506 Vand{\oe}uvre-l{\`e}s-Nancy\\
France\\
\printead{e1}}
\address[B]{A. Kohatsu-Higa\\
Department of Mathematical Sciences\\
Ritsumeikan University\\
\quad and Japan Science and Technology
Agency\\
1-1-1 Nojihigashi\\
Kusatsu, Shiga, 525-8577\\
Japan\\
\printead{e2}}
\address[C]{S. Tindel\\
Institut {\'E}lie Cartan Lorraine\\
Universit\'e de Lorraine\\
B.P. 239\\
54506 Vand{\oe}u\-vre-l{\`e}s-Nancy\\
France\\
\printead{e3}}
\end{aug}
\thankstext{t1}{Supported by grants of the Japanese
government.}

%
\received{\smonth{10} \syear{2013}}
%
\revised{\smonth{9} \syear{2014}}

%
\begin{abstract}
In this paper we obtain Gaussian-type lower bounds for the density
of solutions to stochastic differential equations (SDEs) driven by a
fractional Brownian motion with Hurst parameter $H$. In the
one-dimensional case with additive noise, our study encompasses all parameters
$H\in(0,1)$, while the multidimensional case is restricted
to the case $H>1/2$. We rely on a mix of pathwise methods for
stochastic differential equations and stochastic analysis tools.
\end{abstract}

%
\end{frontmatter}

\section{Introduction}\label{sec1}

Let $B=(B^1,\ldots,B^d)$ be a $d$-dimensional fractional
Brownian motion (fBm in the sequel) defined on a complete
probability space $(\Omega,\mathcal{F},\mathbf{P})$, with Hurst parameter
$H\in(0,1)$. Recall that this means that $B$ is a centered
Gaussian process indexed in $[0,1]$, whose coordinate processes are
independent, and their covariance structure is defined by
%
\begin{eqnarray}
R ( t,s ) &:=& \mathbf{E} \bigl[B_s^j
B_t^j \bigr] =\tfrac{1}{2} \bigl(
s^{2H}+t^{2H}-|t-s|^{2H} \bigr)
\nonumber
\\[-8pt]
\label{eq:exp-cov-fbm}\\[-8pt]
\eqntext{\mbox{for } s,t
\in[0,1] \mbox{ and } j=1,\ldots,d.}
\end{eqnarray}
This implies that the variance of an increment is given by
%
\begin{equation}
\label{eq:var-increm-fbm}
\mathbf{E} \bigl[ \bigl(B_t^j
-B_s^j \bigr)^2 \bigr]= |t-s|^{2H}\qquad
\mbox{for }s,t\in[0,1].
\end{equation}
In particular, this process is $\gamma$-H\"older continuous a.s. for
any $\gamma<H$ and is an $H$-self similar process. This converts fBm
into a natural
generalization of Brownian motion and explains the fact that it
is used in applications \cite{KS,OTHP,WTT}.

We are concerned here with the following class of stochastic
differential equations (SDEs)
in $\mathbb{R}^m$ driven by $B$ on the time interval $[0,1]$:
%
\begin{equation}
\label{eq:sde-intro}
X_t =a +\int_0^t
V_0 (X_s)\,ds+ \sum_{i=1}^d
\int_0^t V_i (X_s)\,dB^i_s,
\end{equation}
where $a\in\mathbb{R}^m$ is a generic initial condition, and
$\{V_i;  0\le i\le
d\}$ is a collection of smooth and bounded vector fields of
$\mathbb{R}^{m}$. Though equation (\ref{eq:sde-intro}) can be solved
thanks to rough paths methods in the general case $H\in(1/4,1)$, $d\ge
1$, we shall
consider in the sequel three situations which can be handled without
recurring to this kind of technique:
\begin{longlist}[(3)]
\item[(1)] The one-dimensional case with additive noise and $H\in(0,1)$, which
can be
treated via simple ODE techniques.

\item[(2)] The one-dimensional situation, namely $m=d=1$ with $H\in(1/2,1)$,
where the
equation can be solved thanks to a Doss--Sussman-type methodology,
as mentioned in~\cite{NSi}.

\item[(3)] The case of a Hurst exponent $H\in(1/2,1)$, for which Young
integration methods are available; see, for example, \cite{Gu,NR,Za}.
\end{longlist}
Hence, we always understand the solution to equation
(\ref{eq:sde-intro}) according to the three settings mentioned
above. We shall see, however, that rough path-type arguments
shall be involved in some of our proofs.

The process defined as the solution of (\ref{eq:sde-intro}) is
obviously worth studying, and a natural step in this direction
is to analyze the density of the random variable $X_t$ for a fixed
$t>0$. In this respect, the following results are available in
our cases of interest:
\begin{longlist}[(2)]
\item[(1)]
For $m=d=1$, the existence of density for $\mathcal{L}(X_t)$ is
examined in \cite{NSi}.
\item[(2)]
Whenever $H>1/2$ and in a multidimensional setting, the
existence of density is established in \cite{NS}, while
smoothness under elliptic assumptions is handled in~\cite{HN}.
\end{longlist}
Let us also mention that for multidimensional equation
(\ref{eq:sde-intro}) and $H\in(1/4,1/2)$, rough path techniques
also enable the study of densities of the solution. We refer to
\cite{CF,CFV} for existence and \cite{CHLT} for smoothness
results for $\mathcal{L}(X_t)$. However, the only Gaussian-type estimate
for the density we are aware of, is the one contained in~\cite{BOT}, which relies heavily on a skew-symmetric assumption
for the vector fields $V_1,\ldots,V_d$.

The current article is thus dedicated to give Gaussian-type
lower bounds for the density of $X_t$. More specifically, we
work under the following assumptions on the coefficients of
equation~(\ref{eq:sde-intro}):

\begin{hypothesis}\label{hyp:vector-fields}
The coefficients $V_0,\ldots,V_d$ of equation
(\ref{eq:sde-intro}) satisfy the following conditions:
\begin{longlist}[(2)]
\item[(1)] If $m=d=1$, then $V_0, V_1\in\mathcal{C}_b^3$, and we also
assume $\lambda\leq|V_1|\leq\Lambda$.
\item[(2)] In the multidimensional case, the vector fields
$V_0,\ldots,V_d$ belong to the space $\mathcal{C}_b^{\infty}$ of smooth
functions bounded together with all their higher order derivatives.
Furthermore, if $V(x)$ denotes the matrix
$(V_1(x),\ldots,V_d(x))\in\mathbb{R}^{m\times d}$ for all $x\in
\mathbb{R}^{m}$,
then we assume the following uniform elliptic
condition:
%
\begin{equation}
\label{eq:elliptic-condition-V}
\lambda \mbox{Id}_{m} \le V(x) V^{*}(x) \le
\Lambda \mbox{Id}_{m} \qquad \mbox{for all } x\in\mathbb{R}^{m},
\end{equation}
where the inequalities are understood in the matrix sense and
where $\lambda$ and $\Lambda$ are two given strictly positive
constants which
are independent of $x$.
\end{longlist}
\end{hypothesis}

With these hypotheses in hand, our main goal is to prove
the following result:

\begin{theorem}\label{thm:low-bnd-intro}
Consider equation (\ref{eq:sde-intro}), under the following three
specific situations:
\begin{longlist}[(III)]
\item[(I)]
$m=d=1$, $H\in(0,1)$, $V_0\in\mathcal{C}_{b}^{1}$ and the noise is additive
(i.e., $V_1$
is a nonvanishing real constant).
\item[(II)] $m=d=1$, $H\in(1/2,1)$ and
Hypothesis \ref{hyp:vector-fields}(1) is satisfied for $V_0,V_1$.
\item[(III)]
Arbitrary $m$, $d\in\mathbb{N}$, $H\in(1/2,1)$ and $V_0,\ldots, V_d$
satisfy Hypothesis~\ref{hyp:vector-fields}(2).
\end{longlist}
Then the solution $X_t$ of equation (\ref{eq:sde-intro})
possesses a density $p_t(x)$ such that for every $x\in\mathbb{R}^m$ and
$t\in(0,1]$, we have
%
\begin{equation}
\label{eq:low-bnd-intro}
p_t(x)\geq \frac{c_1}{t^{mH}} \exp \biggl(-
\frac{c_2\llvert x-a\rrvert ^2}{t^{2H}} \biggr),
\end{equation}
for some constants $c_1,c_2$ only depending on $d,m$ and
$V_0,\ldots,V_d$.
\end{theorem}

As mentioned above, this is (to the best of our knowledge) the
first Gaussian-type lower bound obtained for the density of the solution
of the SDE driven by fBm in a
general setting. It should also be mentioned that lower
bound (\ref{eq:low-bnd-intro}) can be complemented by a similar
upper bound contained in \cite{BNOT}.

Let us say a few words about the methodology we rely on in order
to obtain our lower bound (\ref{eq:low-bnd-intro}). Generally
speaking it is based on Malliavin calculus tools, but the
three results mentioned in Theorem~\ref{thm:low-bnd-intro} are
proved in different ways:
\begin{longlist}[(3)]
\item[(1)] In the one-dimensional additive case, we invoke a
recent formula for densities introduced in \cite{NV} which
yields an easy way to estimate $p_t$ in the case of additive stochastic
equations. We thus include this study for didactical purposes,
and also because we obtain (slightly nonoptimal) Gaussian upper
and lower bounds with elegant methods. Observe that this technique
proves to be useful (generally speaking) for equations with additive
noise, as assessed in a SPDE context in \cite{DNQ}.

\item[(2)] The one-dimensional case with multiplicative noise
is based on the Doss--Sussmann transform and Girsanov-type
arguments. It is rather easy to implement and yields
results when the criterion of \cite{NV} cannot be applied.

\item[(3)]  As far as the general case is concerned, it will be basically
handled, thanks to the decomposition of random variables, using
increments independent of Gaussian increments strategy
introduced in \cite{Ba,Ko}, which has also been invoked successfully,
for example, in \cite{DN}.
However, let us point out two important
differences between the fBm and the diffusion case:
\end{longlist}
\begin{longlist}[(ii)]
\item[(i)]
In the case of SDE (\ref{eq:sde-intro}) without drift
coefficient $V_0$, the first step of the method implemented (for a fixed
$t\in(0,1]$) in \cite{Ba,Ko} amounts to introducing a partition
$\{t_j;  0\le j \le n\}$
such that $t_0=0$ and $t_n=t$, with $n$ large enough, and then
splitting $X_t$ into small contributions of the form
%
\begin{equation}
\label{eq:split-Xt-intro}
\quad X_{t_{j+1}} - X_{t_{j}} = \sum
_{i=1}^d V_i (X_{t_{j}})
\bigl[B^i_{t_{j+1}}-B^i_{t_{j}} \bigr] + \sum
_{i=1}^d \int_{t_{j}}^{t_{j+1}}
\bigl[V_i (X_s) - V_i (X_{t_{j}})
\bigr] \,dB^i_s.\hspace*{-10pt}
\end{equation}
Then a main conditionally Gaussian contribution $V_i (X_{t_{j}})
[ B^i_{t_{j+1}}-B^i_{t_{j}}]$ is identified on the right-hand
side of equation (\ref{eq:split-Xt-intro}), while the other terms
are a small remainder in the Malliavin
calculus sense in comparison with the first. Roughly speaking, Gaussian
lower bound
(\ref{eq:low-bnd-intro}) is then obtained by adding those main
contributions and proving that the remainder does not significantly
modify the estimate. However, let us highlight the fact that
this general scheme does not fit to the fractional Brownian
motion setting.

Indeed, due to the fBm dependence structure, the
main contributions to the variance of $X_t$ in the current
situation come from the cross terms
$\mathbf{E}[(B^i_{t_{j+1}}-B^i_{t_{j}})(B^i_{t_{k+1}}-B^i_{t_{k}})]$
for $j\ne k$. We have thus decided to express equation (\ref{eq:sde-intro}) as
an anticipative Stratonovich-type equation with respect to the
Wiener process induced by $B$. This is known to be an
inefficient way to solve the original equation, but turns out to
be very useful in order to analyze the law of $X_t$. We shall
detail this strategy at Section~\ref{sec:gen-setting-d-dim}.

\item[(ii)]
In the case of an equation driven by usual Brownian motion, the
Malliavin--Sobolev norms involved in the computations give
deterministic contributions after conditioning, due to the independence
of increments of the Wiener process. This is not true, however, in the
fBm case, and we thus need to add a proper localization to the
arguments in \cite{Ba,Ko}.

The adaptation of the Brownian methodology to our fBm context is thus
nontrivial. Note that we could also have tried to resort to the
powerful global bounds given in \cite{MN} in order to get our Gaussian
lower bounds. Unfortunately, the exponential moments conditions imposed
in the latter reference are too restrictive to be applied to Malliavin
derivatives of SDEs driven by fBm.
\end{longlist}

Our article is structured as follows: Section~\ref{sec:stoch-calc-fbm} is devoted to recall some useful facts
on fractional Brownian motion and stochastic differential
equations. We handle the one-dimensional case with additive
noise at Section~\ref{sec:one-dim-additive} and the one-dimensional
case with multiplicative noise in Section~\ref{sec:one-dim-multiplicative} with different methodologies.
Finally, the bulk of our article focuses on the general
multidimensional case contained in Section~\ref{sec:general-low-bound}. Some auxiliary results used in
Section~\ref{sec:general-low-bound} dealing with stochastic derivatives are given
in an \hyperref[sec:app1]{Appendix}.

\begin{notation*}
Throughout this paper, unless otherwise
specified, we use $|\cdot|$ for Euclidean norms and
$\|\cdot\|_{L^p}$ for the $L^p(\Omega)$ norm with respect to the
underlying probability measure $\mathbf{P}$. For a random variable $X$,
$\mathcal{L}(X)$ denotes its law and for a $\sigma$-field $\mathcal
{F}$, $X\in\mathcal{F}$ denotes the fact that $X$ is $\mathcal{F}$-measurable.

Consider a finite-dimensional vector space $V$ and a subset $U\subset
\mathbb{R}^d$. The space of
$V$-valued H\"older continuous functions defined on $U$,
with $k$-deriva\-tives which are $\gamma$-H\"older continuous with
$\gamma
\in(0,1)$, will be denoted
by $\mathcal{C}^{k+\gamma}(U;V)$, or just $\mathcal{C}^{k+\gamma}$
when $U=[0,1]$. For a
function $g\in\mathcal{C}^\gamma(V)$ and $0\le s<t\le
1$, we shall consider the semi-norms
%
\begin{equation}
\label{eq:def-holder-norms}
\|g\|_{s,t,\gamma}=\sup_{s\le u<v\le
t}
\frac{|g_v-g_u|_{V}}{|v-u|^{\gamma}}.
\end{equation}
The semi-norm $\|g\|_{0,1,\gamma}$ will simply be denoted by
$\|g\|_{\gamma}$. Similarly, for an open set $U$, $\mathcal
{C}^1_b(U;V)$ denotes
the space of bounded continuously differentiable functions with bounded
first derivative.
For\vspace*{1pt} $x,y\in\mathbb{R}^m$, we set $\mathbf{1}_{\{y\ge x\}}:=\prod_{k=1}^m\mathbf{1}_{\{
y_k\ge x_k\}}$. Vectors $x\in\mathbb{R}^m$ denote column vectors,
their $j$th component is denoted by $x^j$ and the transpose of $x$ is
denoted by $x^*$. The identity matrix of order $m\times m$ is denoted
by $Id_m$.

Finally, let us mention that generic constants will be denoted
by $c,c_{H},c_{V}$, etc., independently of their actual value
which may change from one line to the next. This rule will also apply
for the constants $M$ and $M'$ which will appear as localization
parameters, with the following additional convention: each time a
localization constant appears, it increases its value by the addition
of a fixed universal constant from the previous value. For a detailed
explanation, see (\ref{Hub}).
\end{notation*}

\section{Stochastic calculus for fractional Brownian motion}
\label{sec:stoch-calc-fbm}

This section is devoted to giving some of the basic elements of stochastic
calculus with respect to $B$.
For some fixed $H\in(0,1)$, we consider $(\Omega,\mathcal{F},\mathbf
{P})$ the
canonical probability space associated with the fractional
Brownian motion (in short fBm) with Hurst parameter $H$. That
is, $\Omega=\mathcal{C}_0([0,1];\mathbb{R}^{d})$ is the Banach space
of continuous
functions
vanishing at $0$ equipped with the supremum norm, $\mathcal{F}$ is the
Borel sigma-algebra and $\mathbf{P}$ is the unique probability
measure on $\Omega$ such that the canonical process
$B=\{B_t=(B^1_t,\ldots,B^d_t),   t\in[0,1]\}$ is a fBm with Hurst
parameter $H$.
In this context, let us recall that $B$ is a $d$-dimensional
centered Gaussian process, whose covariance structure is induced
by equation~(\ref{eq:var-increm-fbm}).

\subsection{Malliavin calculus tools}\label{sec:malliavin-tools}
Gaussian techniques are obviously essential in the
analysis of
fBm driven differential equations like (\ref{eq:sde-intro}), and
we proceed here to introduce some of them; see Chapter~5 in \cite{Nu06} for
further details.

\subsubsection{Wiener space associated to fBm}\label{sec:wiener-space-fbm}
Let $\mathcal{E}$ be the space of $\mathbb{R}^d$-valued step
functions on $[0,1]$, and $\mathcal{H}$ the closure of
$\mathcal{E}$ under the distance defined by the scalar product
\[
\bigl\langle(\mathbf{1}_{[0,t_1]} , \ldots, \mathbf{1}_{[0,t_d]}),(
\mathbf{1}_{[0,s_1]} , \ldots, \mathbf{1}_{[0,s_d]}) \bigr
\rangle_{\mathcal{H}}=\sum_{i=1}^d
R(t_i,s_i).
\]
The space $\mathcal{H}$ is isometric to the reproducing kernel Hilbert
space associated to $B$.

Furthermore, if $(e_1,\ldots,e_d)$ designates the canonical
basis of $\mathbb{R}^d$, one constructs an isometry $K^*$:
$\mathcal{H}\rightarrow L^2([0,1];\mathbb{R}^d)$ such that
$K^*(\mathbf{1}_{[0,t]}  e_{i}) = \mathbf{1}_{[0,t]}$ $K_H(t,\cdot
) e_{i}$,
where the kernel $K=K_H$ is given by
%
\begin{eqnarray}
 K(t,s)&=& c_H s^{{1}/2 -H} \int
_s^t (u-s)^{H-{3}/2} u^{H-{1}/2}
\,du, \qquad H>\frac{1}2,
\nonumber
\\
\label{eq:def-kernel-K}
K(t,s)&=& c_{H,1} \biggl(\frac{s}{t} \biggr)^{1/2-H}
(t-s)^{H-1/2}\\
&&{} + c_{H,2} s^{1/2-H} \int
_s^t (u-s)^{H-{1}/2} u^{H-{3}/2}
\,du, \qquad H<\frac{1}2,
\nonumber
\end{eqnarray}
for $0\le s\le t$ and some explicit universal constants $c_H$,
$c_{H,1}, c_{H,2}$. With a slight abuse of notation we will denote the
associated\vspace*{1pt} integral operator by $Kf(x)=\int_0^x f(s)K(x,s)\,ds$. Note
that we have that $R(s,t)=
\int_0^{s\wedge t} K(t,r)\* K(s,r)  \,dr$. Moreover, let us observe
that $K^*$ can be represented
in the following form: for $H\in(1/2,1)$, we have
\[
\label{eq:def-K-star-regular}
\bigl[K^* \varphi\bigr]_t = \int_t^1
\varphi_r \,\partial_r K(r,t) \,dr
\]
while for $H\in(0,1/2)$ it holds that
\[
\bigl[K^* \varphi\bigr]_t = K(1,t) \varphi_t + \int
_t^1 (\varphi_r-
\varphi_t ) \,\partial_r K(r,t) \,dr.
\]
When $H\in(1/2,1)$ it can be shown that $L^{1/H} ([0,1],
\mathbb{R}^d)
\subset\mathcal{H}$, and when $H\in(0,1/2)$ one has
$\mathcal{C}^\gamma\subset\mathcal{H}\subset L^2([0,1])$
for all $\gamma>\frac{1}{2}-H$. We shall also use the following
representations of the inner product in $\mathcal{H}$:
\begin{longlist}[(ii)]
\item[(i)] For $H\in(1/2,1)$ and $\phi,\psi\in\mathcal{H}$, we have
%
\begin{equation}
\label{eq:inner-pdt-H-smooth}
\quad\bigl\langle K^*\phi, K^*\psi\bigr\rangle_{L^2([0,1])}= \langle
\phi, \psi\rangle_{\mathcal{H}}=c_H\int_0^1
\!\!\int_0^1 | s-t |^{2H-2} \langle
\phi_{s} , \psi_{t}\rangle_{\mathbb{R}^d} \,ds \,dt.
\end{equation}

\item[(ii)] For $H\in(0,1/2)$, consider any family of
partitions $\pi=(t_j)$ of $[0,1]$, and set
$Q_{jk}=\sum_{i=1}^d\mathbf{E}[\Delta^i_j(B)\Delta^i_k(B)]$ with
$\Delta^i_j(B)=B^i_{t_j}-B^i_{t_{j-1}}$. Then for
$\phi,\psi\in\mathcal{H}$, we have
%
\begin{equation}
\label{eq:inner-pdt-H-rough}
\langle\phi, \psi\rangle_{\mathcal{H}} = \lim
_{|\pi|\to0} \sum_{j,k} \langle
\phi_{t_{j-1}} , \psi_{t_{k-1}}\rangle_{\mathbb{R}^d}
Q_{jk}.
\end{equation}

%

Let us also recall that there exists a $d$-dimensional Wiener
process $W$ defined on $(\Omega,\mathcal{F},\mathbf{P})$ such that
$B$ can be
expressed as
%
\begin{equation}
\label{eq:volterra-representation}
B_t=\int_{0}^{t}K(t,r)
\,dW_r, \qquad t\in[0,1].
\end{equation}
This formula will be referred to as Volterra's representation of
fBm.
Formula (\ref{eq:volterra-representation}) has
various important implications. For example, it is readily
checked that $\mathcal{F}_t\equiv\sigma\{B_s;   0\le s\le t\}
=\sigma\{W_s;   0\le
s\le t\}$. This filtration
will appear in the sequel.
\end{longlist}

\subsubsection{Malliavin calculus for $B$}
Isometry arguments allow us to define the Wiener integral
$B(h)=\int_0^{1} \langle h_s, dB_s \rangle$ for any element
$h\in\mathcal{H}$, such that it satisfies $\mathbf{E}[B(h_1)$
$B(h_2)]=\langle h_1,  h_2\rangle_{\mathcal{H}}$ for any
$h_1,h_2\in\mathcal{H}$.
An $\mathcal{F}$-measurable real
valued random variable $F$ is then said to be cylindrical if it
can be
written, for a given $n\ge1$, as
\[
F=f \bigl(B\bigl(h^1\bigr),\ldots,B\bigl(h^n\bigr)
\bigr)= f \biggl( \int_0^{1} \bigl\langle
h^1_s, dB_s \bigr\rangle ,\ldots,\int
_0^{1} \bigl\langle h^n_s,
dB_s \bigr\rangle \biggr),
\]
where $h^i \in\mathcal{H}$ and $f\dvtx \mathbb{R}^n \rightarrow
\mathbb{R}$ is a $C^{\infty}$ bounded function with bounded
derivatives. The set of
cylindrical random variables is denoted by $\mathcal{S}$.

The Malliavin derivative with respect to $B$ is defined as
follows: for $F \in\mathcal{S}$, the derivative of $F$ is the
$\mathbb{R}^d$ valued
stochastic process $(\mathbf{D}_t F )_{0 \leq t \leq1}$ given
by
\[
\mathbf{D}_t F=\sum_{i=1}^{n}
h^i_{t} \frac{\partial
f}{\partial
x_i} \bigl( B\bigl(h^1
\bigr),\ldots,B\bigl(h^n\bigr) \bigr).
\]
More generally, we can introduce iterated derivatives. We will use the
following notation, depending on the situation. For $F \in\mathcal{S}$,
we set for $\mathbf{i}=(i_1,\ldots,i_k)$ and $\mathbf{t}=(t_1,\ldots,t_k)$
\[
\mathbf{D}^k_{\mathbf{t}}=\mathbf{D}^k_{t_1,\ldots,t_k}
F = \mathbf {D}_{t_1} \cdots\mathbf{D}_{t_k} F \quad\mbox{or}\quad
\mathbf{D}^{\mathbf{i}}_{\mathbf
{t}}F=\mathbf{D}^{i_1,\ldots,i_k}_{t_1,\ldots,t_k}
F=\mathbf{D}^{i_1}_{t_1} \cdots\mathbf{D}^{i_k}_{t_k}
F.
\]
For any $p \geq1$, it can be checked that the operator
$\mathbf{D}^k$ is closable from
$\mathcal{S}$ into $L^p(\Omega;\mathcal{H}^{\otimes k})$.
We denote by
$\mathbf{D}^{k,p}$ the closure of the class of
cylindrical random variables with respect to the norm
\[
\| F\| _{k,p}= \Biggl( \mathbf{E} \bigl[
F^{p} \bigr] +\sum_{j=1}^k
\mathbf{E} \bigl[ \bigl\llVert \mathbf{D}^j F\bigr\rrVert
_{\mathcal{H}^{\otimes j}}^{p} \bigr] \Biggr) ^{{1}/{p}},
\]
for $k\geq0$ and $p\geq1$. In particular, $\| F\|
_{0,p}\equiv
\| F\| _{p}
=(\mathbf{E} [ F^{p} ])^{1/p}$.
As it is usually the case in
Malliavin calculus with respect to $W$, the spaces $\mathbf{D}^{k,p}(\mathcal{H})$
are also defined.
The dual operator of
$\mathbf{D}$ is denoted by $\bolds{\delta}$, which corresponds to
the Skorohod integral with respect to the fBm $B$ on the
interval $[0,1]$.
The space of smooth processes $\mathbf{L}^{k,p}(\mathcal{H})$ is induced
by the following norm:
\[
\|u\|_{\mathbf{L}^{k,p}(\mathcal{H})}^{p} = \mathbf{E} \bigl[\| u
\|_{\mathcal{H}}^{p} \bigr]+ \sum_{l=1}^{k}
\mathbf{E} \bigl[\bigl\| \mathbf{D}^{l}u\bigr\| _{\mathcal{H}^{\otimes(l+1)}}^{p}
\bigr].
\]
Finally, the set of smooth integrands is defined as
$\mathbf{D}^{\infty}(\mathcal{H})=\bigcap_{k, p \geq1}
\mathbf{D}^{k,p}(\mathcal{H})$,
and the Malliavin covariance matrix
of $F$ is denoted by $\bolds{\Gamma}_F$.

As mentioned in the \hyperref[sec1]{Introduction}, our lower bound
(\ref{eq:low-bnd-intro}) will be obtained by considering
equation (\ref{eq:sde-intro}) as an equation driven by the
underlying Wiener process $W$ defined in (\ref{eq:volterra-representation}), meaning that we shall also use
stochastic analysis estimates with respect to $W$. We refer
to Chapter~1 in \cite{Nu06} for this classical setting, and just
mention here a some notation: we denote by $\mathrm{D}$ the
differentiation operator with respect to $W$ and by $\delta$ the
corresponding dual operator (Skorohod integral). The respective
norms in the Sobolev spaces $D^{k,p}(L^2([0,1]))$ are denoted by
$\|\cdot\|_{k,p}$ and the space of smooth integrands by
$L^{k,p}$. The following simple relation between $\mathbf{D}$ and
$\mathrm{D}$ is then shown in \cite{Nu06}, Proposition 5.2.1:

\begin{proposition} \label{prop:relation-D-DW}
Let ${D}^{1,2}$ be the Malliavin--Sobolev space corresponding to
the Wiener process $W$. Then $\mathbf{D}^{1,2}=(K^{*})^{-1}{D}^{1,2}$,
and for
any $F\in{D}^{1,2}$ we have $\mathrm{D} F = K^{*} \mathbf{D}F$
whenever both members of the relation are well defined.
\end{proposition}

In fact the above proposition says that the derivatives $\mathbf{D}$ and
$\mathrm{D}$ are somewhat interchangeable. Indeed, using formula (5.14)
in \cite{Nu06}, which gives an explicit formula for $(K^*)^{-1}$, one
obtains such a property. In particular, we will use that for $F\in
\mathcal{F}_t$ with $F\in{D}^{k,p}$ and for $\mathbf{u}=(u_{1},\ldots
,u_{k})\in[0,1]^{k}$ and $\mathbf{r}=(r_1,\ldots,r_n)$, we have
%
\begin{equation}
\label{eq:DWDB}
\bigl|\mathrm{D}^k_{\mathbf{u}}F\bigr|\le\mathop{\operatorname{ess}\operatorname{sup}}_{u_i\le r_i;
i=1,\ldots,k}\bigl|\mathbf{D}^k_{\mathbf{r}}F\bigr|K(t,u_1) \cdots K(t,u_k).
\end{equation}
For the proof of (\ref{eq:DWDB}) and other {useful} properties, see
\hyperref[sec:app1]{Appendix}.

Some of our computations in Section~\ref{sec:general-low-bound}
will rely on some conditional Malliavin calculus arguments, for
which some definitions need to be recalled. First, for a given
$t\in[0,1]$ and $F\in L^{2}(\Omega)$, we shorten notation and
write
\[
\mathbf{E}_t[F]:=\mathbf{E}[F| \mathcal{F}_t],
\]
and also set $\mathbf{P}_t$ for the respective conditional probability
and $\operatorname{Cov}_t(G)$ for the conditional covariance matrix of a
Gaussian vector $G$.
We shall only use conditional Malliavin calculus with respect to the
underlying Wiener process $W$, for which we recall the following
definitions: For a random variable $F$ and $t\in[0,1]$, let $\|F\|
_{k,p,t}$ and
$\Gamma_{F,t}$ be the quantities defined (for $k\ge0$, $p>0$) by
%
\begin{eqnarray}
\|F\|_{k,p,t}&=& \Biggl( \mathbf{E}_t \bigl[
F^{p} \bigr] +\sum_{j=1}^k
\mathbf{E}_t \bigl[ \bigl\llVert {D}^j F\bigr\rrVert
_{(L^{2}_{t})^{\otimes j}}^{p} \bigr] \Biggr)^{{1}/{p}}\quad \mbox{and}
\nonumber
\\[-8pt]
\label{eq:def-conditional-sobolev}
\\[-8pt]
\nonumber
\Gamma_{F,t}&=& \bigl(\bigl\langle F^i, F^j\bigr
\rangle_{L^{2}_{t}} \bigr)_{1\leq i, j\leq d},
\end{eqnarray}
where we have set $L^{2}_{t}\equiv L^{2}([t,1])$.

With this notation in hand, we give a conditional version of the
integration by parts formula with respect to the Wiener process $W$,
borrowed from \cite{Nu06}, Proposition~2.1.4.

\begin{proposition} \label{prop:int-parts-cond}
Fix $n\geq1$. Let $F, Z_s, G\in({D}^{\infty})^d$ be three
random vectors where $Z_s$ is $\mathcal{F}_s$-measurable and
$(\det_{{\Gamma}_{F+Z_s}})^{-1}$ has finite moments of
all orders. Let $g\in\mathcal{C}_p^\infty(\mathbb{R}^d)$. Then, for any
multi-index
$\alpha=(\alpha_1, \ldots,  \alpha_n)\in\{1, \ldots, d\}^n$,
there exists a r.v. ${H}^s_\alpha(F,G)\in\bigcap_{p\geq
1}\bigcap_{m\geq0} {D}^{m,p}$ such that
%
\begin{equation}
\label{eq:int-parts-cond}
\mathbf{E} \bigl[(\partial_\alpha g)
(F+Z_s) G|\mathcal{F}_s \bigr] = \mathbf{E}
\bigl[g(F+Z_s) {H}_\alpha^s(F,G)|
\mathcal{F}_s \bigr],
\end{equation}
where ${H}_\alpha^s(F,G)$ is
recursively defined by
\begin{eqnarray*}
{H}_{(i)}^s(F,G)&=& \sum_{j=1}^d
{\delta}_s \bigl( G \bigl({{\Gamma}}_{F,s}^{-1}
\bigr)_{ij} D F^j \bigr),\\
{H}_{\alpha}^s(F,G) &=& {H}^s_{(\alpha_n)}
\bigl(F,{H}^s_{(\alpha_1, \ldots,
\alpha_{n-1})}(F,G)\bigr).
\end{eqnarray*}
Here ${\delta}_s$ denotes the Skorohod integral with respect
to the Wiener process $W$ on the interval $[s,1]$.
Furthermore, the following norm estimates with $\frac{1}p=\frac{1}{q_1}+\frac{1}{q_2}+\frac{1}{q_3}$ hold true:
\[
\bigl\| {H}_{\alpha}^s(F,G)\bigr\|_{p,s}\le c \bigl\|\det(\Gamma_{F,s})^{-1}\bigr\|^{n}_{2^{n-1}q_1,s}\|F\|_{n+2,2^{n}q_2,s}^{2(dn+1)} \|G\|_{n,q_3,s}.
\]
\end{proposition}

We will also resort to a localized version of the above bounds.
Namely, we introduce a family of functions $\Phi_{M,\epsilon}\dvtx \mathbb
{R}_{+}\to
\mathbb{R}_{+}$ indexed by $M,\epsilon>0$, which
are regularizations of $\mathbf{1}_{\{x\leq M\}}$. Specifically, we
define a function $\phi_\epsilon=\epsilon^{-1}\phi \dvtx \mathbb{R}\to
\mathbb{R}$ with
\[
\phi(x):= c_{\phi} \exp \biggl(-\frac{1}{1-x^2} \biggr)
\mathbf{1}_{\{|x|<1\}},
\]
where $c_{\phi}$ is a normalization constant chosen in order to
have $\int_{\mathbb{R}} \phi(x)  \,dx=1$. Then we define
%
\begin{equation}
\label{eq:def-Phi}
\Phi_{M,\epsilon}(z):=1-\int_{-\infty}^z
\phi_\epsilon ({x-M} )\,dx.
\end{equation}
It is then
{readily checked that $\Phi_{M,\epsilon}(z)=0$ for $z>M+\epsilon$,
$\Phi_{M,\epsilon}(z)=1$ on $[0,  M-\epsilon]$} and $\Phi_{M,\epsilon
}\in
C^{\infty}_{b}$.
We will use the above localization function in two situations: one for
$M\gg 1$, $\epsilon=1$, and
in this case we simplify the notation using $\Phi_M\equiv\Phi_{M,1}$.
In a second case $M$ will not be a large quantity and therefore we will
have to choose $\epsilon$ accordingly.

Consider now $Z\in{D}^\infty$. Under the same conditions as for
Proposition~\ref{prop:int-parts-cond}, we get a conditional integration
by parts formula of form (\ref{eq:int-parts-cond}) localized by $Z$,
with the following modification on the estimation of
the norms of ${H}_{\alpha}^s$:
%
\begin{eqnarray}
&&\quad \bigl\| {H}_{\alpha}^s\bigl(F,G \Phi_M(Z)\bigr)\bigr\|_{p,s}
\nonumber
\\[-8pt]
\label{Hub}\\[-8pt]
\nonumber
&& \quad\qquad \le c \bigl\|\det({\Gamma}_{F,s})^{-1}
\Phi_{M'}(Z)\bigr\|_{p_3,s}^{k_3} \bigl\|F
\Phi_{M'}(Z) \bigr\|_{k_2,p_2,s}^{k_4}\bigl\|G
\Phi_{M'}(Z)\bigr\|_{k_1,p_1,s},
\end{eqnarray}
for some appropriate positive integers $k_1,p_1,k_2,p_2,k_3,p_3,k_4$,
and where we recall our convention on increasing constants $M'>M$. In
fact, to obtain the above inequality is enough to notice that there
exist constants $M'$ and $C$ which may depend on $M$ and $k\in\mathbb
{N}$ such that $\Phi_M(Z)\le C\Phi_{M'}(Z)^k$ and $|\partial^k_z\Phi
_M(Z)|\le C\Phi_{M'}(Z)$. Notice that (\ref{Hub}) is valid for
localizations of the form $\Phi_{M,\epsilon}(Z)$ as well.

\subsection{Differential equations driven by fBm}

Recall that $X$ is the solution of
(\ref{eq:sde-intro}), and that our working assumptions are
summarized in Hypothesis~\ref{hyp:vector-fields}. We have
distinguished 3 situations:
\begin{longlist}[(3)]
\item[(1)]
The one-dimensional additive case, for which equation
(\ref{eq:sde-intro}) can be reduced to an ordinary differential
equation by considering the process $Z=X-B$.
\item[(2)]
The one-dimensional multiplicative case, handled thanks to the
Doss--Sussman transform; see, for example, \cite{NSi}.
\item[(3)]
The multidimensional case with $H\in(1/2,1)$, solved in a
pathwise way by interpreting stochastic integrals as generalized
Riemann--Stieljes-type integrals.
\end{longlist}
In this section we give a brief account on the known results in
the last situation.

In the case $H\in(1/2,1)$, (\ref{eq:sde-intro}) is solved thanks to a fixed
point argument, after interpreting the stochastic integral in
the (pathwise) Young sense; see, for example, \cite{Gu}. Let us recall
that Young's integral can be defined in the following way:

\begin{proposition}\label{prop:young-intg}
Let $f\in\mathcal{C}^\gamma$, $g\in\mathcal{C}^\kappa$ with
$\gamma+\kappa>1$, and
$0\le s\le t\le1$. Then the integral $\int_s^t g_\xi  \,df_\xi$
is well defined as a Riemann--Stieltjes integral. Moreover, the
following estimation is fulfilled:
\[
\biggl\llvert \int_s^t g_\xi
\,df_\xi\biggr\rrvert \leq C \|f\|_\gamma \|g
\|_\kappa|t-s|^\gamma, 
\]
where the constant $C$ only depends on $\gamma$ and $\kappa$.
\end{proposition}

With this definition in mind and under
Hypothesis~\ref{hyp:vector-fields}, we can solve (\ref{eq:sde-intro})
uniquely, in the Young sense. Specifically, it is proven in
\cite{NR} that equation (\ref{eq:sde-intro}) driven by $B$
admits a unique $\gamma$-H\"older continuous solution $X$, for
any $\frac{1}2<\gamma<H$. Moreover, the following moments bounds
are shown in \cite{HN}:

\begin{proposition}\label{prop:moments-sdes}
Let $H\in
(1/2,1)$, and assume that $V_0,\ldots,V_d$ satisfy
Hypothesis \ref{hyp:vector-fields}. Then for $t\in[0,1]$ and
$\frac{1}2<\gamma<H$, we have
%
\begin{equation}
\label{eq:bnd-moments-X-t}
\|X\|_{0,t,\infty} \le|a|+ c_{V} \|B
\|_{0,t,\gamma}^{1/\gamma},
\end{equation}
{where we have set $\|X\|_{0,t,\infty}:=\sup\{|X_{s}| ;   0 \le s
\le
t\}$ and where we recall that $\|B\|_{0,t,\gamma}$ is defined by (\ref
{eq:def-holder-norms}).}
Moreover $X_{t}\in\mathbf{D}^{\infty}$ and for $n\ge1$, $\mathbf
{i}=(i_{1},\ldots,i_{n})\in\{1,\ldots,d\}^{n}$ and $0\le s < t \le1$
the following bound
holds true:
%
\begin{equation}
\label{eq:bnd-moments-D-n-X-t}
\sup_{s\leq u, r_1,\ldots, r_n\leq t} \bigl|\mathbf{D}_{\mathbf
{r}}^{\mathbf{i}}
X_u\bigr| \le C_{V,n} \exp \bigl(c_{V,n} \|B
\|_{s,t,\gamma}^{1/\gamma} \bigr).
\end{equation}
\end{proposition}

We remark that $\mathbf{D}_{r_1,\ldots,r_n}^{i_1,\ldots, i_n} X_u$ is a continuous function except if $r_i=u$
for some $i$, where it is c\`adl\`ag, and therefore the above supremum
is well defined.

Furthermore, a bound for $\gamma$-H\"older norms with
$\frac{1}2<\gamma<H$ is provided in \cite{FV-bk}, equation (10.15), for
$X$ together with its Malliavin derivatives:

\begin{proposition} \label{prop:cotes-hol-sol}
Under the same assumptions as for Proposition~\ref{prop:moments-sdes}, we have
\begin{eqnarray*}
\|X\|_{s,t,\gamma}&\leq &  c_{1,V} \bigl(\|B\|_{s,t,\gamma}\vee \|B
\|_{s,t,\gamma}^{{1}/{\gamma}} \bigr),\\
  \bigl\|\mathbf{D}_{\mathbf{r}}^{\mathbf{i}}
X_u\bigr\|_{s,t,\gamma} &\leq &  c_{2,V,n}\exp
\bigl(c_{3,V,n}\|B\|_{s,t,\gamma}^{{1}/{\gamma}} \bigr).
\end{eqnarray*}
\end{proposition}

\begin{remark}\label{rmk:young-strato}
Assume $H>1/2$ and the other hypothesis of Proposition~\ref{prop:moments-sdes} again.
As\vspace*{1pt} mentioned, for example, in \cite{Co}, Section~7, the Young-type
integrals $\int_0^t V_i (X_s) \,dB^i_s$ in (\ref{eq:sde-intro}) coincide
with the Russo--Vallois definition of integral and also with the
Stratonovich integral of Malliavin calculus. We shall use these
identifications later on, and they will be detailed in Section~\ref{sec:fractional-strato-eq}. For the time being, let us just stress the
following fact: in order to harmonize notation, we shall often write
$\int_0^t V_i (X_s) \circ dB^i_s$ for the Young integral (instead of
$\int_0^t V_i (X_s) \,dB^i_s$), in order to recall that it can also be
interpreted in the Stratonovich sense.
\end{remark}

\section{One-dimensional additive case}
\label{sec:one-dim-additive}

This section is devoted to prove our main Theorem~\ref{thm:low-bnd-intro} in the particular case $m=d=1$ with
additive noise. In this context, one can take advantage of the
results obtained by Nourdin and Viens in \cite{NV} in order to
derive Gaussian-type upper and lower bounds for $p_t$. Let us
then first recall what those results are.

\subsection{General bounds on densities of one-dimensional
random variables}

Recall that we denote the Malliavin--Sobolev spaces with respect to the
fBm $B$ by $\mathbf{D}^{k,p}$, and consider a real-valued centered random variable
$F\in\mathbf{D}^{1,2}$. We define a function $g$ on $\mathbb{R}$ by
\[
g(z):=\mathbf{E} \bigl[ \bigl\langle\mathbf{D} F,-\mathbf{D}
\mathbf{L}^{-1}F \bigr\rangle_\mathcal{H}|F=z \bigr],
\]
where the operator $\mathbf{L}$ is the Ornstein--Uhlenbeck operator
associated to the fBm~$B$ (see \cite{Nu06} for further details),
which can be defined using the chaos expansion by the formula
$\mathbf{L}=-\sum_{n=0}^\infty n\mathbf{J}_n$.
Based on the function $g$, the following simple criterion for
Gaussian-type bounds has been obtained in \cite{NV}:

\begin{proposition}\label{prop:criterion-nourdin-viens}
Let $F\in\mathbf{D}^{1,2}$ with $\mathbf{E}[F]=0$.
If there exist $c_1$, $c_2 >0$ such that
%
\begin{equation}
\label{eq:bnd-gF}
c_1\leq g(F)\leq c_2, \qquad \mathbf{P}\mbox{-a.s.},
\end{equation}
then the law of $F$ has a density $\rho$ satisfying, for almost
all $z\in\mathbb{R}$,
\[
\frac{\mathbf{E}[|F|]}{2c_2} \exp \biggl(-\frac{z^2}{2c_1} \biggr)\leq \rho(z)\leq
\frac{\mathbf{E}[|F|]}{2c_1} \exp \biggl(-\frac{z^2}{2c_2} \biggr).
\]
\end{proposition}

Interestingly enough, Nourdin and Viens \cite{NV}, Proposition 3.7, also
give an
alternative formula for $g(F)$ which is suitable for
computational purposes. Indeed, if we write $\mathbf{D}
F=\Phi_F(B)$ in the above Proposition, where $\Phi_F\dvtx \mathbb
{R}^\mathcal{H}\rightarrow
\mathcal{H}$ is a
measurable mapping, then the following relation holds true:
%
\begin{equation}
\label{eq:def-g-F}
g(F)=\int_0^\infty
e^{-\theta} \mathbf{E} \bigl[ \bigl\langle \Phi_F(B),
\Phi_F\bigl(e^{-\theta} B+\sqrt{1-e^{-2\theta}}B'
\bigr) \bigr\rangle_\mathcal{H} | F \bigr]\,d\theta,
\end{equation}
where $B'$ stands for an independent copy of $B$, and is such
that $B$ and $B'$ are defined on the product probability space
$(\Omega\times\Omega',\mathcal{F}\otimes\mathcal{F}',\mathbf
{P}\times\mathbf{P}')$. Here we
abuse the notation by letting $\mathbf{E}$ be the mathematical
expectation with respect to $\mathbf{P}\times\mathbf{P}'$, while
$\mathbf{E}'$ is
the mathematical expectation with respect to $\mathbf{P}'$ only.
One can thus recast relation (\ref{eq:def-g-F}) as
%
\begin{equation}
\label{eq:g-F}
g(F)=\int_0^\infty \mathbf{E}
\bigl[\mathbf{E}' \bigl[ \bigl\langle\mathbf{D} F,{\mathbf{D}
{F^\theta}} \bigr\rangle_\mathcal{H} \bigr] | F \bigr]\,d\theta,
\end{equation}
where, for any random variable $X$ defined in
$(\Omega,\mathcal{F},\mathbf{P})$, ${X^\theta}$ denotes the
following shifted
random variable in $\Omega\times\Omega'$:
\[
{X^\theta}\bigl(\omega,\omega'\bigr)=X \bigl(
e^{-\theta}\omega+\sqrt{1-e^{-2\theta}}\omega' \bigr),\qquad
\omega\in\Omega, \omega'\in\Omega'.
\]

\subsection{Main result in the additive one-dimensional case}
Before stating our result let us point out that we assume throughout
this subsection $V_1\equiv\sigma$. That is, $X$ is the solution of
%
\begin{equation}
\label{eq:sde-additive}
X_t =x +\int_0^t
V_{0}(X_s)\,ds+ \sigma B_t, \qquad t\in[0,1],
\end{equation}
where $\sigma>0$ is a strictly positive constant, $V_{0}$
satisfies $\llVert V_{0}'\rrVert _\infty\leq M$ for some constant
$M>0$ and $B$ is a fBm with
$H\in(0,1)$. Under this setting, we are able to get the
following bounds:

\begin{theorem}\label{thm:up-low-bnd-additive-case}
Assume that $V_{0}$ satisfies that
$\llVert V_{0}'\rrVert _\infty\leq M$, for some constant $M>0$,
$\sigma>0$ and $H\in(0,1)$. Then, for all $t\in(0,1]$, $X_t$
possesses a
density $p_t$, and there exist some strictly positive
constants $c_1<c_{3}$ and $c_2<c_{4}$ depending only on $M$ and $H$ such
that for all $z\in\mathbb{R}$,
%
\begin{equation}
\label{eq:final-up-low-bnd-additive-case}
\frac{c_{1}}{\sigma t^{H}} {\exp \biggl(-\frac{(z-m)^2}{c_{2}\sigma^2t^{2H}} \biggr)} \leq
p_t(z)\leq \frac{c_{3}}{\sigma t^{H}} \exp \biggl(-\frac{(z-m)^2}{c_{4}\sigma^2t^{2H}}
\biggr).
\end{equation}
\end{theorem}

\begin{remark}
The advantage of the Nourdin--Viens method of estimating
densities is that upper and lower bounds are obtained with similar
proofs. The drawback is the restriction to one-dimensional additive
situations. Also notice that the exponents
in equation~(\ref{eq:up-low-bnd-additive-case}) are optimal, meaning
that our density bounds mimic the fBm case. See also Theorem~\ref{th:4.2} for the nonconstant diffusion case.
\end{remark}

\begin{pf*}{Strategy of the proof}
We first notice that we can reduce our problem to prove that
%
\begin{equation}
\label{eq:up-low-bnd-additive-case}
\quad\frac{\mathbf{E}[|X_t-m|]}{c_{1}\sigma^2t^{2H}} {\exp \biggl(-\frac{(z-m)^2}{c_{2}\sigma^2t^{2H}} \biggr)} \leq
p_t(z)\leq \frac{\mathbf{E}[|X_t-m|]}{c_{2}\sigma^2t^{2H}} \exp \biggl(-\frac{(z-m)^2}{c_{1}\sigma^2t^{2H}}
\biggr).\hspace*{-12pt}
\end{equation}
Indeed, one can check in our context that $\mathbf{E}[|X_t-m|]\asymp
\sigma
t^{H}$. This easy
step is left to the reader for the sake of conciseness, and it
naturally allows us to go from (\ref{eq:up-low-bnd-additive-case}) to~(\ref{eq:final-up-low-bnd-additive-case}). Now in order to prove (\ref
{eq:up-low-bnd-additive-case}), we obviously rely heavily on Proposition~\ref{prop:criterion-nourdin-viens}.
We thus define
$F=X_t-\mathbf{E}[X_t]$, where $X_t$ is the solution of
(\ref{eq:sde-additive}).
We get a centered random variable, and
we shall prove that there exists two constants $0< K_1 <K_2$ such that
%
\begin{equation}
\label{eq:bnd-g-F}
K_1 \sigma^{2} t^{2H} \leq g(F)
\leq K_2 \sigma^{2} t^{2H}.
\end{equation}

Notice first that in the present case, it is easily
seen that for any $t>0$, we have $X_t\in\mathbf{D}^{1,2}$; this is a
particular case of \cite{NS}. Furthermore, the Malliavin
derivative of $X_t$ satisfies the following equation for $r\leq
t$:
\[
\mathbf{D}_r X_t=\int_r^t
V_{0}'(X_s)\mathbf{D}_rX_s\,ds+
\sigma.
\]
This equation can be solved explicitly, and we obtain
%
\begin{equation}
\label{derivada}
\mathbf{D}_r X_t=\sigma
e^{\int_r^t V_{0}'(X_s)\,ds}.
\end{equation}
In particular, the bound
%
\begin{equation}
\label{eq:low-up-bnd-DX}
\sigma e^{-t M} \le\mathbf{D}_r
X_t \le\sigma e^{t M}
\end{equation}
holds true almost surely for $M=\llVert V_{0}'\rrVert _\infty$.

Observe that we shall bound $g(F)$ thanks to relation (\ref{eq:low-up-bnd-DX}).
More specifically, we will show that for each
$\theta\in\mathbb{R}_{+}$ we have (almost surely)
%
\begin{equation}
\label{eq:low-up-bnd-DF-DF-theta}
c_{3} t^{2H} \sigma^{2} \le \bigl
\langle\mathbf{D} F,{\mathbf{D} {F^\theta}} \bigr\rangle
_\mathcal{H} \le c_{4} t^{2H} \sigma^{2},
\end{equation}
for two strictly positive constants $c_{3}<c_{4}$.
This deterministic bound easily yields~(\ref{eq:bnd-gF}) and
thus~(\ref{eq:up-low-bnd-additive-case}). We now separate the cases
$H\in(1/2,1)$
and $H\in(0,1/2)$ in order to get relation~(\ref{eq:low-up-bnd-DF-DF-theta}). Notice that the Brownian case, that is,
$H=1/2$, is well known, and it is thus omitted here for the sake of conciseness.
\end{pf*}

\subsection{Case $H>\frac{1}{2}$}
Recall that we wish to prove (\ref{eq:low-up-bnd-DF-DF-theta}) thanks
to relation (\ref{eq:low-up-bnd-DX}). Furthermore, owing to expression
(\ref{eq:inner-pdt-H-smooth}) for the inner product in $\mathcal{H}$, we
can write $\langle\mathbf{D} F,{\mathbf{D} {F^\theta}}\rangle
_\mathcal{H}$ as
\begin{eqnarray}
\bigl\langle\mathbf{D} F,{\mathbf{D} {F^\theta}} \bigr\rangle
_\mathcal{H} &=&c_H \int_0^t
\!\!\int_0^t \mathbf{D}_uX_t
\mathbf{D}_v{X^\theta_t}|u-v|^{2H-2}
\,du\,dv
\nonumber
\\[-8pt]
\label{g1}
\\[-8pt]
\nonumber
&=&c_H\sigma^2 \int_0^t
\!\!\int_0^t e^{\int_u^t
V_{0}'(X_s)\,ds}e^{\int_v^t
V_{0}'(X^\theta_s)\,ds}|u-v|^{2H-2}
\,du\,dv.
\nonumber
\end{eqnarray}
Therefore the lower and upper bounds in (\ref{eq:low-up-bnd-DX}) follow
from plugging inequality
(\ref{eq:low-up-bnd-DX}) into relation~(\ref{g1}).

\subsection{Case $0<H<\frac{1}{2}$}
As\vspace*{1pt} in the case $H> \frac{1}2$, our aim is to prove (\ref{eq:low-up-bnd-DX}). We thus go
back to equation~(\ref{eq:g-F}), and we observe that we can
reduce the problem to the existence of two constants $0<c_1<c_2$
such that
%
\begin{equation}
\label{eq:bounds-inner-prod-rough}
c_1t^{2H}\leq \bigl\langle\mathbf{D}
X_t, \mathbf{D} X^\theta_t \bigr\rangle_{\mathcal{H}}\leq c_2t^{2H}.
\end{equation}
The proof of these inequalities will rely on the following
quadratic programming lemma, which is a slight variation of
\cite{CHLT}, Lemma 6.2:

\begin{lemma}\label{lem:optim-quadratic}
Let $Q\in\mathbb{R}^{n}\otimes\mathbb{R}^n$ be a strictly positive
symmetric matrix such
that $\sum_{j=1}^{n}Q_{ij}\ge0$ for all $i=1,\ldots,n$. For two
positive constants $a$ and $b$, consider the sets
$\mathcal{A}=[a,\infty)^n$ and $\mathcal{B}=[b,\infty)^n$. Then
\[
\inf \bigl\{x^* Q \tilde{x}; \tilde{x}\in\mathcal{A}, x\in \mathcal{B} \bigr\} =
ab \sum_{i,j=1}^{n} Q_{ij} .
\]
\end{lemma}

\begin{pf}
Set $\mathbf{a}=a  \mathbf{1}\in\mathbb{R}^n$ and $\mathbf{b}=b
\mathbf{1}\in\mathbb{R}^n$. The
Lagrangian of our quadratic programming problem is a function
$L \dvtx  \mathbb{R}^n\times\mathbb{R}^n\times\mathbb{R}_+^n\times
\mathbb{R}_+^n\to\mathbb{R}$ defined as
\[
L(x,\tilde{x},\lambda_1,\lambda_2)= x^{*} Q
\tilde{x} - \lambda _{1}^{*} (x- \mathbf{b} )-
\lambda_{2}^{*} (\tilde{x}- \mathbf{a}).
\]
It is readily checked that $\nabla_x L(x,\tilde{x},\lambda_1,\lambda_2)=
Q\tilde{x}-\lambda_1$ and $\nabla_{\tilde{x}}
L(x,\tilde{x},\lambda_1,\lambda_2)= Qx-\lambda_2$, which vanishes for
$x=Q^{-1}\lambda_2$ and $\tilde{x}=Q^{-1}\lambda_1$. Therefore,
\begin{eqnarray*}
\inf \bigl\{L(x,\tilde{x},\lambda_1,\lambda_2); x,
\tilde {x}\in \mathbb{R}^n \bigr\} &=&L \bigl({Q^{-1}
\lambda_2}, {Q^{-1}\lambda_1},
\lambda_1, \lambda_2 \bigr)
\\
&=&-{\lambda_{1}^{*} Q^{-1}\lambda_2}
+ \lambda_{1}^{*} \mathbf{b}+ \lambda_{2}^{*}
\mathbf{a}=: G(\lambda_1, \lambda_2).
\end{eqnarray*}
We have thus obtained a dual problem of the form
%
\begin{equation}
\label{eq:dual-G}
\max \bigl\{G(\lambda_1, \lambda_2);
\lambda_1, \lambda_2\in \mathbb{R}_+^n \bigr\}.
\end{equation}

Let us now solve Problem (\ref{eq:dual-G}). We first maximize
$G$ without positivity constraints on $\lambda_1$ and $\lambda_2$: we
get $\nabla_{\lambda_1} G(\lambda_1, \lambda_2)=-Q^{-1}\lambda
_2+\mathbf{b}$ and\break
$\nabla_{\lambda_2} G(\lambda_1, \lambda_2)=-\lambda
_{1}^{*}Q^{-1}+\mathbf{a}$, which
vanishes for $\lambda_{1}^{\circ}=Q\mathbf{a}$ and $\lambda
_{2}^{\circ}=Q\mathbf{b}$.
Observe now
that our assumption $\sum_{j=1}^{n}Q_{ij}\ge0$ for all
$i=1,\ldots,n$ implies $\lambda_{1}^{\circ}, \lambda_{2}^{\circ
}\ge0$, so that
$\lambda_{1}^{\circ}$ and~$\lambda_{2}^{\circ}$ are feasible for
the dual
problem. Hence
\[
\max \bigl\{G(\lambda_1, \lambda_2);
\lambda_1, \lambda_2\in \mathbb{R}_+^n \bigr
\} = G\bigl(\lambda_{1}^{\circ}, \lambda_{2}^{\circ}
\bigr) = ab \sum_{i,j=1}^{n}
Q_{ij},
\]
which completes the proof.
\end{pf}

Importantly enough, Lemma~\ref{lem:optim-quadratic} can be
applied in order to get a lower bound on $\mathcal{H}$ norms:

\begin{proposition} \label{prop:lower-bound}
Let $B$ be a one-dimensional fBm on $[0,\tau]$, let
$\mathcal{H}\equiv\mathcal{H}_{\tau}$ be the associated reproducing kernel
Hilbert space and $f, \tilde{f}\in\mathcal{H}$ such that $f_u\ge b$ and
$\tilde{f}_u\ge a$ for any $u\in[0,\tau]$. Then $\langle
f,\tilde{f}\rangle_{\mathcal{H}}\ge a  b   \tau^{2H}$.
\end{proposition}

\begin{pf}
Recall that, owing to relation (\ref{eq:inner-pdt-H-rough}), we
have $\langle f,\tilde{f}\rangle_{\mathcal{H}}=\break\lim_{|\pi|\to
0}I_{\pi}(f, \tilde{f})$, where $\pi$ stands for a generic
partition $\{0=t_0<\cdots< t_n=\tau\}$ and
\[
I_{\pi}(f,\tilde{f}) = \sum_{i,j = 1}^{n}
f_{t_{i-1}} Q_{ij} \tilde{f}_{t_{j-1}} \qquad  \mbox{with }
Q_{ij}=\mathbf{E}\bigl[\Delta_i (B)\Delta_j(B)
\bigr],
\]
where we recall that $\Delta_{i}(B)=B_{t_{i}}-B_{t_{i-1}}$.
We assume for the moment that $Q$ satisfies the hypothesis of
Lemma~\ref{lem:optim-quadratic}, and we get
\[
I_{\pi}(f,\tilde{f}) \ge ab \sum_{i,j = 1}^{n}
Q_{ij} = ab \sum_{i,j = 1}^{n}
\mathbf{E} \bigl[\Delta_i(B)\Delta_j(B) \bigr] = ab
\mathbf{E} \bigl[B_{\tau}^{2} \bigr] = ab \tau^{2H},
\]
which is our claim.

Let us now prove that $Q$ satisfies the hypothesis of Lemma~\ref{lem:optim-quadratic}.
First, the strict positivity of $Q$ stems from the local
nondeterminism of $B$; see, for example, \cite{Xi}. Indeed, for $u\in
\mathbb{R}^n$
we have
\[
u^{*}Qu=\operatorname{Var}  \Biggl( \sum_{j=0}^{n-1}
u_j \Delta_j(B) \Biggr) \ge c_n \sum
_{j=1}^{n} u_j^{2}
|t_{j}-t_{j-1}|^{2H},
\]
where the lower bound is the definition of local
nondeterminism. Thus $u^{*}Qu>0$ as long as $u\ne0$.

Let us now check that for a fixed $i$ we have $\sum_{j=1}^{n}
Q_{ij} \ge0$. To this end, write
\[
\sum_{j=1}^{n} Q_{ij} =
\mathbf{E} \bigl[\Delta_i(B) B_{\tau} \bigr] =\int
_{t_{i}}^{t_{i+1}} \partial_{u} R(\tau,u) \,du.
\]
Going back to expression (\ref{eq:exp-cov-fbm}), it is now
easily seen that for $u<\tau$ we have
\[
\partial_{u} R(\tau,u) = H \bigl(u^{2H-1} + (
\tau-u)^{2H-1} \bigr)>0,
\]
which completes the proof.
\end{pf}

We can now go back to the proof of relation
(\ref{eq:bounds-inner-prod-rough}), which is divided again into
two steps:
\begin{longlist}[Step 2:]
\item[\textit{Step 1:}] \textit{Lower bound}.
Thanks to relation (\ref{eq:low-up-bnd-DX}), we have that
$\sigma e^{-t M}\leq\mathbf{D}_r X_t$.
Thus we just have to apply Proposition~\ref{prop:lower-bound} to
the Malliavin derivative in order to obtain
%
\begin{equation}
\label{eq:low-bnd-DXt-tilde-DXt}
\bigl\langle\mathbf{D} X_t, \mathbf{D}
{X^\theta_t} \bigr\rangle_\mathcal{H}\geq
\sigma^2 t^{2H}e^{-2 M},
\end{equation}
which is our desired lower bound.

\item[\textit{Step 2:}]  \textit{Upper bound}.
In order to obtain an upper bound for $g(F)$, we will use the
representation of $\mathcal{H}$ through fractional derivatives.
Indeed, apply
first the Cauchy--Schwarz inequality in order to get
%
\begin{equation}
\label{eq:cauchy-schwarz-DX}
\bigl\langle \mathbf{D}X_t,{\mathbf{D}X^\theta_t}
\bigr\rangle_\mathcal{H}\leq \llVert \mathbf{D}X_t\rrVert
_{\mathcal{H}} \bigl\|{\mathbf{D}X^\theta _t}
\bigr\|_\mathcal{H}.
\end{equation}
We then invoke Lemma~\ref{lem:bnd-H-C-gamma} to bound $\|\mathbf{D}X^\theta_t\|_\mathcal{H}$. This boils down to estimating
\[
a = \sup_{r\in[0,t]} \bigl|\mathbf{D}_{r}X^\theta_t\bigr|\quad
\mbox{and} \quad b = \sup_{r,v\in[0,t]} \frac{\mathbf{D}_{r}X_{t}^{\theta}-\mathbf
{D}_{v}X_{t}^{\theta
}}{(v-r)^{\gamma}},
\]
with $1/2-H<\gamma<1/2$ and any $\theta\ge0$.

Now starting from expression (\ref{derivada}) and owing to the fact
that $V'_{0}$ is uniformly bounded by $M$, we trivially get $a\le
\sigma
  e^{M}$. As far as $b$ is concerned, we write
\[
\bigl\llvert \mathbf{D}_{r}X_{t}^{\theta}-
\mathbf{D}_{v}X_{t}^{\theta} \bigr\rrvert \le\sigma
e^{\int_v^t V_{0}'(X^\theta_s)\,ds}\bigl\llvert 1- e^{\int_r^v
V_{0}'(X^\theta_s)\,ds}\bigr\rrvert \le\sigma M
e^{2M} (v-r).
\]
We thus end up with the inequalities
\[
a\le\sigma e^{M}\quad \mbox{and}\quad b \le\sigma M e^{2M}
t^{1-\gamma}.
\]
We now apply Lemma~\ref{lem:bnd-H-C-gamma} with constants $a$ and $b$,
and we obtain
\[
 \|\mathbf{D}X_t\|_{\mathcal{H}} \le c_{H} \bigl(
\sigma e^{M} t^{H} + \sigma M e^{2M}
t^{1+H} \bigr) \le2c_{H} \sigma M e^{2M}
t^{H},
\]
and hence
\[
\bigl\langle \mathbf{D}X_t,{\mathbf{D}X^\theta_t}
\bigr\rangle_\mathcal{H}\leq 4c_{H} \sigma^{2}
M^{2} e^{4M} t^{2H}.
\]
Finally, putting together the last bound and
(\ref{eq:low-bnd-DXt-tilde-DXt}), we get (\ref{eq:bnd-g-F}) in
the case $H\in(0,1/2)$, which completes the proof of Theorem~\ref{thm:up-low-bnd-additive-case}.
\end{longlist}

\section{One-dimensional nonvanishing diffusion coefficient case}
\label{sec:one-dim-multiplicative}

We turn now to the case $m=d=1$, $H\in(\frac{1}2, 1)$ for a nonconstant elliptic
coefficient $\sigma$. Observe that this special case is treated in a
separate section because (i) the Gaussian bound is obtained with weaker
conditions on the coefficients than in the multidimensional case, and
(ii) the proof is shorter due to specific one-dimensional techniques
based on the Doss--Sussman transform and Girsanov's theorem. This is
detailed below.

\begin{remark}
The Doss--Sussman transform can be justified for any $H\in(0, 1)$ in
our context. However, the computations related to Girsanov's transform
become much more involved when $H<1/2$, and this is why we restrict our
analysis to $H>1/2$ in the sequel.
\end{remark}

\subsection{Doss--Sussmann transformation}
The idea of the method is to first consider a one-dimensional equation
of Stratonovich-type without drift and then apply Girsanov's theorem
for fBm in order to obtain a characterization of the density.

In order to carry out this strategy, we start by using an independent
copy of $(\Omega,\mathcal{F},\mathbf{P})$ called
$(\Omega',\mathcal{F}',\mathbf{P}')$ supporting a fBm denoted by
$B'$. On
$(\Omega',\mathcal{F}',\mathbf{P}')$, let $Y$ be the unique solution to
%
\begin{equation}
\label{eq:non-vanishing}
Y_t= a+\int_0^t
V_1(Y_s)\circ dB'_s,
\end{equation}
where the integral is interpreted either in the Young or Stratonovich
sense (as recalled in Remark~\ref{rmk:young-strato}),
and where $V_1 \in\mathcal{C}^1(\mathbb{R};\mathbb{R})$, $V_1\neq
0$ and
$H\in(\frac{1}2,1)$. We also call $W'$ the underlying Wiener process
appearing in the Volterra-type representation (\ref{eq:volterra-representation}) for $B'$.
We now recall here some details from Doss and Sussmann's
classical computations adapted to our fBm context.

Indeed, as in \cite{NSi}, let us recall that the solution of
equation (\ref{eq:non-vanishing}) can be expressed as
$Y_t=F(B'_t,  a)$, $t>0$, where $F:\mathbb{R}^2\rightarrow\mathbb{R}$
is the flow associated to $V_1$,
%
\begin{equation}
\label{eq:dif-eq-F}
\frac{\partial F}{\partial x}(x,y)= V_1\bigl(F(x,y)\bigr),
\qquad F(0,y)=y.
\end{equation}
We remark that if $V_1$ is bounded, then $F$ satisfies $|F(x,y)|\le
c(1+|x|+|y|)$.

Next we relate the solution $X$ of equation (\ref{eq:sde-intro}) to the
process $Y$ defined by (\ref{eq:non-vanishing}). This step is partially
borrowed from \cite{NO}, and we refer to that paper for further
details. Indeed, thanks to a Girsanov-type transform, the following
characterization of the law of the solution to (\ref{eq:sde-intro}) is
shown for $m=d=1$: For any bounded measurable function $U:\mathbb
{R}\to\mathbb{R}$, one has
%
\begin{equation}
\label{eq:tail-probab-with-girsanov}
\mathbf{E}_{\mathbf{P}} \bigl[U(X_{t}) \bigr] =
\mathbf{E}_\mathbf{P'} \bigl[U \bigl(F\bigl({B'_t},
a\bigr) \bigr) \xi \bigr],
\end{equation}
where $\xi\equiv\xi_t=\frac{d\mathbf{P}}{d\mathbf{P}'}$ is the
random variable
defined by
%
\begin{equation}
\label{eq:defxi}
\xi =\exp \biggl(\int_{0}^{t}
\biggl[\mathcal{M}_{s} \,dW'_s -
\frac{1}2 \mathcal{M}_{s}^{2} \,ds \biggr] \biggr),
\end{equation}
where we have set $\mathcal{M}= K^{-1}( \int_0^\cdot
V_0V_1^{-1}(Y_u)\,du )$.

Notice that in definition (\ref{eq:defxi}), the operator $K$ has been
alluded to in Section~\ref{sec:wiener-space-fbm}. It should be observed
that $K, K^{-1}$ can also be defined, respectively, for $H\geq\frac{1}2$
and an appropriate function $h$, by (see details in \cite{Nu06}, Chapter~5)
\begin{eqnarray*}
K (h) (s)&=&  I_{0^+}^1 \bigl(s^{H-{1}/2}
\bigl(I_{0^+}^{H-{1}/2} \bigl(s^{{1}/2-H}h\bigr)\bigr)\bigr)
(s) \quad\mbox{and}\\
 K^{-1} (h) (s) &=& s^{H-{1}/2}\bigl(D_{0^+}^{H-{1}/2}
\bigl(s^{{1}/2-H}h'\bigr)\bigr) (s).
\end{eqnarray*}
We also recall that in the last equation, $I_{0^+}^\alpha$ and
$D_{0^+}^\alpha$ denote the
fractional integral and fractional derivative, whose expressions are
\[
I^\alpha_{0^+} f(x)= \frac{1}{\Gamma(\alpha)}\int
_a^x (x-y)^{\alpha-1} f(y)\,dy
\]
and
\[
D_{0^+}^\alpha f(x)=\frac{1}{\Gamma(1-\alpha)} \biggl(
\frac{f(x)}{x^\alpha}+\alpha\int_a^x
\frac{f(x)-f(y)}{(x-y)^{\alpha+1}}\,dy \biggr).
\]
It is easily seen from the expressions of $K_{H}^{-1}$ and
$D_{0^+}^{H-{1}/2}$ that $K_{H}^{-1}h$ is an adapted transformation;
see also expression (\ref{eq:exp-Ms}) below. Hence the term $\xi$ in
(\ref{eq:tail-probab-with-girsanov}) corresponds to the usual Girsanov
correction term.
Furthermore, notice that in order for~(\ref{eq:tail-probab-with-girsanov}) to be satisfied, it is required that
$\int_0^\cdot V_0V_1^{-1}(Y_u)\,du\in I_{0+}^{H+{1}/2}(L^2[0,1])$.
This condition is satisfied due to the $\gamma$-H\"olderianity of $Y$
for any $\gamma<H$.

Actually one should prove that Novikov-type conditions are satisfied
for $\xi$ in order to apply Girsanov's transform and get relation
(\ref{eq:tail-probab-with-girsanov}). This is achieved in the
following lemma:

\begin{lemma}
Let $\xi$ be the random variable defined by (\ref{eq:defxi}), and
assume that Hypothesis \ref{hyp:vector-fields}(1) is satisfied. Then
%
\begin{equation}
\label{eq:bnd-KH-1}
\mathcal{M}_{s} \le c_{V}
\beta_{s}\qquad  \mbox{with } \beta_{s} := s^{{1}/2-H}+
\bigl\|B'\bigr\|_{H-{1}/2+\varepsilon},
\end{equation}
for any arbitrarily small $\varepsilon>0$.
Furthermore $\mathbf{E}_{\mathbf{P}'}[\xi]=1$, which justifies the Girsanov
identity~(\ref{eq:tail-probab-with-girsanov}). That is, under $\mathbf{P}$, $B=B'+\int_0^\cdot V_0V_1^{-1}(Y_u)\,du$ is a $H$-fBm.
\end{lemma}

\begin{pf}
According to the expression of $K_{H}^{-1}$, we have
%
\begin{equation}
\label{eq:exp-Ms}
\mathcal{M}_{s} =\frac{1}{\Gamma(H-{1}/2)} \bigl(
\mathcal{M}_{s}^{1} + \bigl(H-\tfrac{1}{2}\bigr)
\mathcal{M}_{s}^{2} \bigr),
\end{equation}
where we have set
\begin{eqnarray*}
\mathcal{M}_{s}^{1}& \equiv & \frac{V_0V_1^{-1}(Y_s)}{s^{H-{1}/2}},
\\
\mathcal{M}_{s}^{2} &\equiv &  s^{H-{1}/2} \int
_0^s\frac{s^{{1}/2-H}V_0V_1^{-1}(Y_s)-u^{{1}/2-H}{V_0V_1^{-1}(Y_u)}}{(s-u)^{H+{1}/2}}\,du.
\end{eqnarray*}
The term $\mathcal{M}_{s}^{1}$ is easily bounded: we invoke the uniform
ellipticity of $V_{1}$ and the regularity of $V_{0}$ and $V_1$, which
yields $\mathcal{M}_{s}^{1} \le c   s^{-(H-{1}/2)}$. We now bound
$\mathcal{M}
_{s}^{2}$: let us decompose this term as $\mathcal{M}_{s}^{2}=\mathcal
{M}_{s}^{21}+\mathcal{M}
_{s}^{22}$, with
\begin{eqnarray*}
\mathcal{M}_{s}^{21} &=&  \int_0^s
\frac{1-(s/u)^{H-{1}/2}}{(s-u)^{H+{1}/2}} V_0V_1^{-1}(Y_u)
\,du \quad \mbox{and}\\
 \mathcal{M}_{s}^{22} & =&  \int
_0^s\frac{V_0V_1^{-1}(Y_s)-{V_0V_1^{-1}(Y_u)}}{(s-u)^{H+{1}/2}}\,du.
\end{eqnarray*}
Then, resorting again to the fact that $V_0V_1^{-1}$ is bounded and
with the obvious change of variable $r=u/s$, we get
\[
\bigl|\mathcal{M}_{s}^{21}\bigr| \le \frac{c_{V}}{s^{H-1/2}} \int
_{0}^{1} \frac{r^{H-1/2}-1}{r^{H-1/2}
(1-r)^{H+1/2}}\,dr \le
\frac{c_{V,H}}{s^{H-1/2}}.
\]
In order to handle the term $\mathcal{M}_{s}^{22}$, we start by writing
\[
\mathcal{M}_{s}^{22} \leq c_{V} \int
_0^s \frac{|F(B'_s,a)-F(B'_u,a)|}{(s-u)^{H+{1}/2}}\,du,
\]
and thanks to the Lipschitz properties of $F$ plus elementary integral
computations, we obtain
\[
\mathcal{M}_{s}^{22} \le c_{V,H}
\bigl\|B'\bigr\|_{H-{1}/2+\varepsilon}.
\]
Therefore, summarizing our estimates on $\mathcal{M}^{1}, \mathcal
{M}^{21}$ and $\mathcal{M}^{22}$, the proof of our claim (\ref{eq:bnd-KH-1}) is now completed.

Now let us have a closer look at the process $\beta$: it is
readily checked that $\|B'\|_{\gamma}$ admits quadratic
exponential moments for any $\gamma<H$; see Theorem~3 in \cite{NO}. In
particular, one can choose $\gamma=H-1/2+\varepsilon$ for
$\varepsilon$ small enough, and
hence there exists $\lambda>0$ such that the expected value
$\mathbf{E}[\exp(\lambda\int_{0}^{t} \beta^2(s)\,ds)]$ is a finite quantity.
Owing to a version of Novikov's condition stated in~\cite{Fr75}, Theorem~1.1,
we deduce that $\mathbf{E}[\xi]=1$. This completes the proof.
\end{pf}

\subsection{Main result in the Doss--Sussman framework}
As in the additive case of Section~\ref{sec:one-dim-additive},
we are able to get both upper and lower Gaussian bounds in a
one-dimensional context:

\begin{theorem}
\label{th:4.2}
Assume that $H\in(1/2,1)$ and $V_0, V_1$ satisfy the assumptions of
Hypothesis~\ref{hyp:vector-fields}(1). Then there exist constants
$C_1$ and $C_2$ such that for all $t\in(0,1]$, the
solution $X_t$ to equation (\ref{eq:sde-intro}) possesses a
density $p_t$ satisfying for all $x\in\mathbb{R}$,
%
%
\begin{eqnarray}
\quad \frac{1}{C_1\sqrt{2\pi t^{2H}}} {\exp \biggl(-C_1
\frac{(x-a)^2}{2t^{2H}} \biggr)}&\leq& p_t(x)
\nonumber
\\[-8pt]
\label{eq:lowb}\\[-8pt]
\nonumber
&\leq& \frac{1}{C_2\sqrt{2\pi t^{2H}}} \exp
\biggl(-C_2\frac{(x-a)^2}{2t^{2H}} \biggr).
\end{eqnarray}
\end{theorem}

\begin{pf}
In this proof one should separate 4 cases: (a) $\lambda\leq V_1(z)\leq
\Lambda$ with subcases $x>a$ and $x\le a$ and (b) $-\Lambda\leq
V_1(z)\leq-\lambda$ with subcases $x<-a$ and $x\ge-a$. These
situations are treated thanks to the same kind of arguments, and we
will thus assume in the proof that $x\ge a$ and $\lambda\leq
V_1(z)\leq
\Lambda$ for all $z\in\mathbb{R}$. We now divide our proof in two steps.

\begin{longlist}[Step 1:]
\item[\textit{Step 1:}]  \textit{Upper bound}.
We start from an equivalent of
(\ref{eq:tail-probab-with-girsanov}) for densities, which is justified
by~\cite{HN}, Theorem~7,  and a duality argument
%
\begin{equation}
\label{eq:density-with-girsanov}
p_{t}(x)= \mathbf{E}_\mathbf{
\mathbf{P}'} \bigl[\delta_{x} \bigl(F
\bigl({B'_t}, a\bigr) \bigr) \xi \bigr],
\end{equation}
where $\xi$ is the random variable defined in (\ref{eq:defxi}).
We now integrate by parts in order to get
\[
p_{t}(x)= \mathbf{E}_{\mathbf{P}'} \bigl[\mathbf{1}_{\{F(B'_t, a)\geq x\}}
H \bigl( F\bigl(B'_t, a\bigr), \xi \bigr) \bigr],
\]
with
%
\begin{equation}
\label{eq:def-H-F-xi}
H \bigl(F\bigl(B'_t, a\bigr), \xi \bigr)
= \delta \biggl(\frac{\xi  \mathrm{D} F(B'_t, a)}{\llVert  \mathrm{D}
F(B'_t, a) \rrVert _{L^{2}([0,t])}^{2}} \biggr),
\end{equation}
where $\mathrm{D},\delta$, respectively, stand (with a slight abuse of
notation) for the Malliavin
derivative and divergence operator for the Brownian motion
$ W'$ under ${\mathbf{P}'}$.
Let us further simplify the expression for the random variable
$H( F(B'_t, a), \xi)$: setting $K_{t}(u)\equiv
K(t,u)\mathbf{1}_{[0,t]}(u)$, it is readily checked that we have
\[
\mathrm{D}_{u} F\bigl(B'_t, a\bigr) =
\partial_{x} F\bigl(B'_t, a\bigr)
K_{t}(u)\quad  \mbox{and}\quad \bigl\llVert \mathrm{D} F\bigl(B'_t,
a\bigr) \bigr\rrVert _{L^{2}([0,t])}^{2} = \bigl|\partial_{x}
F\bigl(B'_t, a\bigr)\bigr|^{2} t^{2H}.
\]
Plugging this information into (\ref{eq:def-H-F-xi}), and
defining $Z:= \xi  (\partial_{x} F(B'_t, a))^{-1}$,
we end up with
\[
H \bigl(F\bigl(B'_t, a\bigr), \xi \bigr) =
\frac{\delta (Z   K_{t}  )}{t^{2H}} = K_{1}-K_{2},
\]
where
\[
K_{1}=\frac{Z   B'_{t} }{t^{2H}}\quad \mbox{and}\quad K_{2}=
\frac{  \langle\mathrm{D}Z ,   K_{t}  \rangle
_{L^{2}([0,t])}}{t^{2H}}.
\]
We have thus obtained
%
\begin{equation}
\label{eq:dcp-pt1-pt2}
\hspace*{6pt}p_{t}(x)= \mathbf{E}_{\mathbf{P}'} [
\mathbf{1}_{\{F(B'_t, a)\geq x\}} K_{1} ] - \mathbf{E}_{\mathbf{P}'} [
\mathbf{1}_{\{F(B'_t, a)\geq x\}} K_{2} ] =: p_{t}^{1}(x)-p_{t}^{2}(x),
\end{equation}
and we shall upper\vspace*{1pt} bound these two terms separately.

The term $p_{t}^{1}(x)$ can be bounded as follows: for
$q_1,q_2,q_3>1$ large enough and a parameter $1<q_4= 1+\varepsilon$ with
an arbitrarily small $\varepsilon>0$, we have
%
\begin{eqnarray}
p_{t}^{1}(x) &\le& \frac{\mathbf{E}_{{\mathbf{P}'}}^{1/q_{1}} [| B'_t |^{q_{1}}
 ]}{t^{2H}} {
\mathbf{P}'}^{1/q_{2}} \bigl(F\bigl(B'_t,
a\bigr)\geq x \bigr)
\nonumber
\\[-8pt]
\label{eq:dcp-pt1-xi1-xi4}\\[-8pt]
\nonumber
&&{}\times\mathbf{E}_{{\mathbf{P}'}}^{1/q_{3}} \bigl[\bigl|
\partial_{x}F\bigl(B'_t, a\bigr)
\bigr|^{-q_{3}} \bigr] \mathbf{E}_{{\mathbf{P}'}}^{1/q_4} \bigl[
\xi^{q_4} \bigr].
\end{eqnarray}
We now bound the right-hand side of this inequality:

\begin{longlist}[(iii)]
\item[(i)]  We obviously have $\frac{\mathbf{E}_{{\mathbf
{P}'}}^{1/q_{1}} [|
B'_t |^{q_{1}}  ]}{t^{2H}}\leq c t^{-H}$, since
${B'}$ is a $\mathbf{P}'$-fBm.

\item[(ii)]
Let us prove that there exist two positive constants $c_1$ and $c_2$
such that, for all $x\ge0$,
%
\begin{equation}
\label{eq:bounds-inverse-F}
\mathbf{P}'^{1/q_{2}} \bigl(F
\bigl(B'_t, a\bigr)\geq x \bigr)\le c_{1}
\exp \biggl(-\frac{c_{2}(x-a)^2}{t^{2H}} \biggr).
\end{equation}
Indeed, for a fixed $a\in\mathbb{R}$, set $Q\equiv{\mathbf{P}'}(
F(B'_t, a)\geq x )$,
and decompose this term as $Q=Q_{1}+Q_{2}$ with
\[
Q_{1} = \mathbf{P}' \bigl(F\bigl(B'_t,
a\bigr)\geq x, B'_{t} \ge0 \bigr)\quad \mbox{and}\quad
Q_{2} = \mathbf{P}' \bigl(F\bigl(B'_t,
a\bigr)\geq x, B'_{t} < 0 \bigr).
\]
Since we have assumed $x>a$ and $V_{1} > \lambda>0$, it is readily checked
that $Q_{2}=0$.
In the sequel we thus bound the term $Q_{1}$. Toward this aim,
appealing to relation~(\ref{eq:dif-eq-F}), we write
\[
Q_{1} = {\mathbf{P}'} \biggl(\int_{0}^{B'_{t}}
V_{1} \bigl(F(z,a) \bigr)\,dz \geq x -a , B'_{t}
\ge0 \biggr).
\]
Next recall that we have assumed $\lambda\leq V_1(z)\leq\Lambda$ for
all $z\in\mathbb{R}$. Hence we have $\int_{0}^{\zeta} V_{1}
(F(z,a)  )\,dz$
$\le$ $\Lambda\zeta$ for all $\zeta\ge0$, and thus
\[
Q_{1} \le{\mathbf{P}'} \bigl(\Lambda
B'_{t} \geq x -a , B'_{t} \ge0
\bigr) = {\mathbf{P}'} \bigl(\Lambda B'_{t}
\geq x -a \bigr) \le\exp \biggl(-\frac{(x-a)^2}{\Lambda^{2}t^{2H}} \biggr),
\]
which is consistent with relation (\ref{eq:bounds-inverse-F}). The
proof is now completed by a similar analysis of the term $Q_{2}$.

\item[(iii)]
Equation (\ref{eq:dif-eq-F}) and the nondegeneracy assumptions on
$V_{1}$ show that $\partial_{x}F$ is bounded from below by a constant,
so that we get the trivial bound
\[
\mathbf{E}_{{\mathbf{P}'}}^{1/q_{3}} \bigl[\bigl| \partial_{x}F
\bigl(B'_t, a\bigr) \bigr|^{-q_{3}} \bigr]\le c.
\]

\item[(iv)] Set $S=\int_{0}^{t} \mathcal{M}_{s}   \,dW'_s$ and
$D=\int_{0}^{t}
\mathcal{M}_{s}^{2}   \,ds$, where $\mathcal{M}\equiv K^{-1}_H( \int_0^\cdot
V_0\times V_1^{-1}(Y_u)\,du )$ as above, and where we recall that
$q_{4}=1+\varepsilon$
with an
arbitrarily small $\varepsilon>0$. It is readily checked that
\[
\xi^{q_{4}} = \exp \biggl(q_{4} S-\frac{q_{4}}{2} D
\biggr) = \exp \biggl(q_{4} S-\frac{q_{4}^{2}}{2} D \biggr) \exp
\biggl(\frac{q_{\varepsilon}}{2} D \biggr),
\]
where $q_{\varepsilon}=q_{4}^{2}-q_{4}=\varepsilon(1+\varepsilon)$.
Now observe that the
term $\exp( q_{4} S-\frac{q_{4}^{2}}{2} D)$ is a Girsanov
change of measure which corresponds to a shift on ${B'}$ of the
form
\[
\hat{B}={B}'-q_{4}\int_0^\cdot
V_0V_1^{-1}(Y_u)\,du
=B-(q_{4}-1)\int_0^\cdot
V_0V_1^{-1}(Y_u)\,du.
\]
Calling $\hat{{\mathbf{P}'}}$ the probability under which $\hat{B}$
is a fBm, we get
%
\begin{equation}
\label{eq:expect-hat-Q-KH-1}
\mathbf{E}_{{\mathbf{P}'}} \bigl[\xi^{q_4} \bigr] =
\mathbf{E}_{\hat{{\mathbf{P}'}}} \biggl[\exp \biggl(\frac
{q_{\varepsilon}}{2} D \biggr)
\biggr].
\end{equation}
Now plug estimate (\ref{eq:bnd-KH-1})
into (\ref{eq:expect-hat-Q-KH-1}). This yields
\begin{eqnarray*}
D &\le& c_{V} \bigl(1+ \bigl\|B'\bigr\|^2_{H-{1}/2}
\bigr)
\\
&\le& c_{V} \biggl(1+ \biggl\|\hat{B}+q_{4} \int
_0^\cdot V_0V_1^{-1}(Y_u)
\,du \biggr\|_{H-{1}/2 }^{2} \biggr)
\\
&\le& c_{V} \bigl(1+ \|\hat{B}\|_{H-{1}/2 }^{2}
\bigr).
\end{eqnarray*}
Going back to relation (\ref{eq:expect-hat-Q-KH-1}) and taking
into account the fact that $q_{\varepsilon}$ can be chosen arbitrarily
small, we get $\mathbf{E}_{{\mathbf{P}'}}[ \xi^{q_4} ]<\infty$.
\end{longlist}

Gathering all the above estimates into
(\ref{eq:dcp-pt1-xi1-xi4}), we have thus obtained that
\[
p_{t}^{1}(x) \le \frac{c_{1}}{t^{H}} \exp \biggl(-
\frac{c_{2}(x-a)^2}{t^{2H}} \biggr).
\]
The upper bound for $p_{t}^{2}(x)$ [defined in (\ref{eq:dcp-pt1-pt2})]
is obtained along the same lines, and we spare the details to the
reader. Let us just mention that more Malliavin derivatives of $\xi$
and $F(B',a)$ are involved in the computations, and this is where we
use both the nondegeneracy and smoothness assumptions on $V$. Then
taking into account the estimates on $p_{t}^{1}(x)$ and $p_{t}^{2}(x)$
in~(\ref{eq:dcp-pt1-pt2}), we end up with our global upper bound in
(\ref{eq:lowb}).

\item[\textit{Step 2:}] \textit{Lower bound}.
Our strategy to obtain the lower bound in (\ref{eq:lowb}) is based on
the following decomposition:
%
\begin{equation}
\label{eq:decA}
\qquad p_t(x)=\mathbf{E}_\mathbf{P'}
\bigl[\delta_x\bigl(F\bigl({B'_t}, a\bigr)
\bigr) (\xi _t-\xi_{c_1t}) \bigr] +\mathbf{E}_\mathbf{P'}
\bigl[\delta_x\bigl(F\bigl({B'_t}, a\bigr)
\bigr) \xi_{c_1t} \bigr] =:\rho_{t}^{1} +
\rho_{t}^{2},
\end{equation}
where $c_1$ is a constant to be determined later. Observe that the main
term will be $\rho_{t}^{2}$, which means that we consider a two-point
partition of the interval $[0,t]$, and we perform a one-step
decomposition of $X_{t}$ (or $Y_{t}$) on $[0,c_{1}t]$ and $[c_{1}t,t]$,
as opposed to the general time interval partition in Section~\ref{sec:general-low-bound}.

First, we start studying the main term $\rho_t^2$: Note that due to
(\ref{eq:volterra-representation}), we can apply Girsanov's theorem in
order to get
\begin{eqnarray*}
\rho_{t}^{2}&=& \mathbf{E}_\mathbf{P'}
\bigl[\mathbf{E}_{\mathbf{P'}} \bigl[\delta _x\bigl(F
\bigl({B'_t}, a\bigr)\bigr) |\mathcal{F}_{c_1t}
\bigr] \xi_{c_1t} \bigr]
\\
&=&\mathbf{E}_\mathbf{P'} \biggl[\exp \biggl(-
\frac{(F^{-1}(x,a)-\int_0^{c_1t}K(t,s)\,dW'_s)^2}{2\int_{c_1t}^tK^{2}(t,s)\,ds} \biggr)\frac
{\partial
_x F^{-1}(x,a)}{\sqrt{2\pi\int_{c_1t}^tK^{2}(t,s)\,ds}} \xi _{c_1t} \biggr]\\
&=&
\mathbf{E} _\mathbf{P}[L_{c_1,t}],
\end{eqnarray*}
where we have set
\begin{eqnarray*}
L_{c_1,t}&:=&  \exp \biggl(-\frac{(F^{-1}(x,a)-\int_0^{c_1t}K(t,s)\,dW_s+\int_0^{c_1t}V_0V_1^{-1}(X_s)\,ds)^2}{2\int_{c_1t}^tK^{2}(t,s)\,ds} \biggr)\\
&&{}\times\frac
{\partial_x F^{-1}(x,a)}{\sqrt{2\pi\int_{c_1t}^tK^{2}(t,s)\,ds}} .
\end{eqnarray*}
In order to determine a lower bound for the above expression, we use
the following information:
\begin{longlist}[(iii)]
\item[(i)]  We have $\partial_xF^{-1}(x,a){ \ge[V_{1}(F(x,a))]^{-1} } \ge
\Lambda^{-1}$.
\item[(ii)]  We apply the inequality $(m+a)^2\ge\frac{1}2 m^2- 2a^2$ to
$m\equiv
F^{-1}(x,a)-\int_0^{c_1t}K(t,s)\,dW_s$ and $a$ defined by $a^2\equiv
(\int_0^{c_1t}V_0V_1^{-1}(X_s)\,ds )^2\le c_V t^2$.
\item[(iii)]  Gaussian convolution identities can be invoked in order to
compose the quadratic exponential term defining $L_{c_1,t}$ with the
expected value with respect to the Gaussian random variable $\int_0^{c_1t}K(t,s)\,dW_s$.
\item[(iv)]  The following trivial bound holds true: $\int_{c_{1}t}^{t}
K^{2}(t,s)   \,ds \le\int_{0}^{t} K^{2}(t,s)   \,ds = t^{2H}$.
These ingredients easily entail that
\[
\rho_t^2\ge\frac{c}{\sqrt{2\pi\hat{\sigma}^2}} \exp \biggl(-
\frac{F^{-1}(x,a)^2}{2\hat{\sigma}^2} \biggr),
\]
for $\hat{\sigma}^2=2\int_{c_1t}^t K^{2}(t,s)\,ds+\int_0^{c_1t}K^{2}(t,s)\,ds$, and we observe that $\sigma^{2}\le\hat
{\sigma}^2\le2\sigma^{2}$.
\end{longlist}
\end{longlist}

Now we estimate the first term $\rho_{t}^{1}$ in (\ref{eq:decA}) and
prove that it is upper bounded by a quantity which is smaller than half
of the lower bound we have just obtained. For this term we need to use
again the integration by parts estimates carried out in~(\ref{eq:density-with-girsanov}). In order not to repeat arguments we just
mention the main steps: we start by writing
\[
\rho_{t}^{1} = \mathbf{E}_\mathbf{P'}
\bigl[\delta_x\bigl(F\bigl({B'_t}, a\bigr)
\bigr) (\xi_t-\xi _{c_1t}) \bigr] = \mathbf{E}_{\mathbf{P}'}
\bigl[\mathbf{1}_{\{F(B'_t, a)\leq x\}} H \bigl( F\bigl(B'_t,
a\bigr), \xi_t-\xi_{c_1t} \bigr) \bigr],
\]
and we decompose this expression into $p^1-p^2$ like in (\ref{eq:dcp-pt1-pt2}), except for the fact that this time $Z$ is replaced by
$Z_t:= ((\xi_t-\xi_{c_1t} ) \,\partial_{x} F(B'_t, a))^{-1}$.

We wish to take advantage of the fact that $\xi_t-\xi_{c_1t}$ is a
small quantity whenever $c_{1}$ is close to 1. For this, define the
process $\mathcal{M}_{c_{1}t, \cdot}$ as $\mathcal{M}_{c_{1}t,
s}=K^{-1}_H(\int_{c_1t}^\cdot V_0V_1^{-1}(Y_u)\,du)$, consider $\theta\in[0,1]$ and define
\[
\xi_t(\theta):=\xi_{c_1t} \exp \biggl(\theta\int
_{c_1t}^t \mathcal{M}_{c_{1}t, s}
\,dW'_s -\frac{\theta^2}2 \int_{c_1t}^t
\mathcal{M}_{c_{1}t, s}^2 \,ds \biggr).
\]
Then by the mean value theorem, we have
\[
\xi_t-\xi_{c_1t} =\int_0^1
\,d\theta \xi_t(\theta) \biggl(\int_{c_1t}^t
\mathcal{M}_{s} \,dW'_s-\theta\int
_{c_1t}^t \mathcal{M}_{s}^2
\,ds \biggr).
\]
Applying Fubini's theorem, one sees that the same estimates as in (\ref
{eq:dcp-pt1-xi1-xi4}) appear again with the following exceptions: (i)
The last term in the decomposition becomes $\mathbf{E}_{\mathbf
{P'}}^{1/q_4} [
(\xi_t(\theta))^{q_4}  ]$, which is handled in the same fashion as
before. (ii) There is another term appearing in the decomposition, namely
\[
\mathbf{E}_{\mathbf{P'}}^{1/q_5} \biggl[ \biggl(\int
_{c_1t}^t \mathcal{M}_{s}
\,dW'_s -\theta\int_{c_1t}^t
\mathcal{M}_{s}^2 \,ds \biggr)^{q_5} \biggr].
\]
Using\vspace*{1.5pt} (\ref{eq:bnd-KH-1}) and {the same estimates for stochastic
integrals as in step 1}, one obtains that the latter term is upper
bounded by $c(1-c_1^{2-2H})t^{2-2H}$. Therefore taking $c_1$
sufficiently close to 1 one obtains that this upper bound is smaller
than $1/2$ of the lower bound previously obtained.
The proof is now complete.
\end{pf}

\section{General lower bound}
\label{sec:general-low-bound}
We now wish to obtain Gaussian-type lower bounds for the
multi-dimensional case of equation (\ref{eq:sde-intro}).
However, the computations in this section will be performed on
the following simplified version for notational sake (adaptation
of our calculations to the drift case are straightforward):
%
\begin{equation}
\label{eq:sde-Vd}
X_t =a +\sum_{i=1}^d
\int_0^t V_i(X_s)
\circ dB_s^ i,
\end{equation}
where $a\in\mathbb{R}^m$ is a generic initial condition,
$V_i\dvtx \mathbb{R}^m\rightarrow\mathbb{R}^m$ $i=1,\ldots,  d$ is a
collection of
smooth and bounded vectors fields and $B^1, \ldots,  B^d$ are
$d$ independent fBm's with $H\in(1/2,1)$.
Recall that our goal is then to prove relation
(\ref{eq:low-bnd-intro}) in this context. To this end, we shall
assume that Hypothesis \ref{hyp:vector-fields} [especially relation
(\ref{eq:elliptic-condition-V})] is satisfied for
the remainder of the article.
Observe that, as in Section~\ref{sec:one-dim-multiplicative}, equation
(\ref{eq:sde-Vd}) is written in the Stratonovich sense. Relations
between Stratonovich and Young integrals will be investigated in
Section~\ref{sec:fractional-strato-eq}.

\subsection{Preliminary considerations}
\label{sec:gen-setting-d-dim}
Let us recall briefly the strategy used in \cite{Ba,Ko} in
order to obtain Gaussian lower bounds for solutions of stochastic
differential equations. The argument starts with some additional
notation: Recall that the natural filtration of $B$, which is also the
natural filtration of the underlying Wiener process $W$ defined
by (\ref{eq:volterra-representation}), is denoted by $\mathcal{F}_t$. As
we have introduced in Section~\ref{sec:malliavin-tools}, we
write $\mathbf{E}_{t}$ for the conditional expectation with respect to
$\mathcal{F}_t$. Under our working Hypothesis \ref{hyp:vector-fields},
let us also mention that the following result is available (see
\cite{BNOT,HN} for further details):

\begin{proposition}
Under Hypothesis \ref{hyp:vector-fields}, there exists a
unique solution to~(\ref{eq:sde-Vd}). Then for any
$t\in(0,1]$, the random variable $X_t$ is nondegenerate in the
sense of Definition~2.1.1 in \cite{Nu06}, namely: \textup{(i)} $X_{t}\in\mathbf
{D}^{\infty}$; \textup{(ii)} the Malliavin matrix $\Gamma_{X_{t}}$ is almost
surely invertible and satisfies $\Gamma_{X_{t}}^{-1}\in\bigcap_{p\ge1}
L^{p}(\Omega)$. In particular, the density of $X_t$ admits
the representation $p_t(x)=\mathbf{E}[\delta_x(X_t)]$, where $\delta_x$
stands for the Dirac measure at point $x$.
\end{proposition}

With this preliminary result in hand, the quantity
$\mathbf{E}[\delta_x(X_t)]$ will be analyzed by means of the succesive
evaluation of conditional densities of an
approximation sequence $\{F_j;  0\le j \le n\}$ such that
$X_t=F_n$. We thus consider $p_t(x)=\mathbf{E}[\delta_x(F_n)]$. The
discretization procedure is based on a corresponding partition
of the time interval as $\pi: 0=t_0<\cdots<t_n=t$, and the sequence of
random variables
$F_j$ which satisfy the relation $F_j\in\mathcal{F}_{t_j}$.

Let us give some hints about the general strategy for the
discretization: it is designed to take advantage of conditional
Malliavin calculus, which allows one to capture the
convolution property of Gaussian distributions. We shall thus
assume for the moment a structure of the form
%
\begin{equation}
\label{eq:dcp-Fj-I-Rd}
F_j= F_{j-1} + I_j +
R_j,
\end{equation}
where we recall that $F_{j-1}\in\mathcal{F}_{t_{j-1}}$. In formula
(\ref{eq:dcp-Fj-I-Rd}), the term $I_j$ will stand for a
Gaussian random variable (conditionally to $\mathcal{F}_{t_{j-1}}$), and
$R_j$ refers to a small remainder term, whose contribution to
the density of $F_j$ can be neglected with respect to the one
induced by $I_j$ just like in the argument in (\ref{eq:decA}). The
local Gaussian bound
(\ref{eq:low-bnd-intro}) will be obtained from the density of the
sum $\sum_{j=1}^{n}I_j$. The argument will finish by an application of
the Chapman--Kolmogorov formula.

As suggested by equation (\ref{eq:split-Xt-intro}) {and setting
$\Delta
^i_{j+1}(B):=B_{t_{j+1}}^{i}-B_{t_{j}}^{i}$}, a natural
candidate consists of taking $F_j=X_{t_{j}}$, which yields
%
\begin{equation}
\label{eq:first-Ij-Rj}
\quad I_{j}=\sum_{i=1}^d
V_i (X_{t_{j}}) \Delta^i_{j+1}(B) \quad
\mbox{and} \quad R_{j}=\sum_{i=1}^d
\int_{t_{j-1}}^{t_{j}} \bigl[V_i
(X_s) - V_i (X_{t_{j}}) \bigr]
\,dB^i_s.\hspace*{-6pt}
\end{equation}
However, this simple and natural guess is not suitable for the
fBm case. Indeed, the analysis of the variances of $I_j$ induced
from decomposition (\ref{eq:first-Ij-Rj}) reveals that a significant
amount is generated by the covariances between the increments $\Delta
^i_{j} (B)$. Now, if we write
%
\begin{equation}
\label{eq:dcp-variance-fbm}
t^{2H} = \mathbf{E} \bigl[ \bigl(B_{t}^{i}
\bigr)^2 \bigr]= \mathbf{E} \Biggl[ \Biggl( \sum
_{j=1}^{n} \Delta^i_j(B)
\Biggr)^2 \Biggr] = \sum_{j,k=1}^{n}
\mathbf{E} \bigl[\Delta^i_j(B)\Delta^i_k(B)
\bigr],
\end{equation}
we realize that the diagonal terms on the right-hand side
expression only account for a term of the form $\sum_{j}
|t_{j}-t_{j-1}|^{2H}$, which vanishes as the mesh of the
partition goes to 0 when $H\in(1/2,1)$. This means that our
decomposition (\ref{eq:first-Ij-Rj}) will not be able to capture
the correct amount of variance contained in $X_t$, and has to be
modified.

There are at least two natural generalizations of the Euler-type method
described above:
\begin{longlist}[(2)]
\item[(1)]
Take into account the off-diagonal terms in
(\ref{eq:dcp-variance-fbm}), and perform a block type analysis.
\item[(2)]
Express the equation as an equation driven by the Wiener process
$W$ defined by relation (\ref{eq:volterra-representation}), and
take advantage of the independence of the increments of $W$.
\end{longlist}
We have not been able to implement the strategy (1) above without
cumbersome calculations, and we have thus chosen to follow the second
approach. Toward this aim, we first recall how to define
equation~(\ref{eq:sde-Vd}) as a Stratonovich equation with
respect to $W$.

\subsection{Fractional equations as Stratonovich-type equations}\label
{sec:fractional-strato-eq}
In order to handle equation~(\ref{eq:sde-Vd}) as an equation
with respect to $W$, let us first introduce the following
functional space:

\begin{definition}
Let $|\mathcal{H}|$ be the space of measurable functions
$\phi:[0,1]\rightarrow\mathbb{R}^{d}$ such that
\[
\|\phi\|_{|\mathcal{H}|}^2 := \alpha_H \int
_0^1 \biggl(\int_0^1
|\phi_r||\phi_u||r-u|^{2H-2}\,dr \biggr)
\,du<+\infty.
\]
Note that $|\mathcal{H}|$ endowed with the norm $\|\cdot\|_{|\mathcal
{H}|}$ is a Banach
space of functions, which is also a subspace of $\mathcal{H}$.
\end{definition}

In the sequel we also consider random elements with values in
$|\mathcal{H}|$. In particular, the norm of $\phi$ in $\mathbf
{D}^{1,2}(|\mathcal{H}|)$
is given by
\[
\|\phi\|_{\mathbf{D}^{1,2}(|\mathcal{H}|)}=\mathbf{E} \bigl[\|\phi\| _{|\mathcal{H}|}^2
\bigr]+\mathbf{E} \bigl[\|\mathbf{D} \phi\|_{|\mathcal{H}|\otimes
|\mathcal{H}|}^2 \bigr].
\]

As mentioned before, the Young-type integrals we have handled so far
can be identified with Stratonovich-type integrals with respect to $B$,
and finally as anticipative Stratonovich-type integrals with respect to
$W$. In order to state these results more formally, let us recall what
we mean by Stratonovich integrals with respect to $B$:

\begin{definition}\label{def:strato-intg}
Let $u=\{u_t,  t\in[0,1]\}$ be a $\mathbb{R}^{d}$-valued process
defined on
$(\Omega,\mathcal{F},\mathbf{P})$, whose paths are supposed to be
integrable. The
Stratonovich (or symmetric, or Russo--Vallois) integral of $u$ with
respect to $B$ is denoted by $\sum_{k=1}^{d}\int_{0}^{1} u_{s}^{k}
\circ dB_{s}^{k}$ and is defined as
\[
\sum_{k=1}^{d}\int_{0}^{1}
u_{s}^{k} \circ dB_{s}^{k} = \lim
_{\varepsilon\to0} \frac{1}{2\varepsilon} \sum_{k=1}^{d}
\int_{0}^{1} u_{s}^{k}
\bigl(B_{s+\varepsilon}^{k} - B_{s-\varepsilon}^{k} \bigr)
\,ds,
\]
whenever the limit exists. In the same way, the indefinite Stratonovich
integral is defined as
%
\begin{equation}
\label{eq:def-indefinite-strato-intg}
\sum_{k=1}^{d}\int
_{0}^{t} u_{s}^{k} \circ
dB_{s}^{k} = \sum_{k=1}^{d}
\int_{0}^{1} \bigl(u_{s}^{k}
\mathbf {1}_{[0,t]}(s) \bigr)\circ dB_{s}^{k}\qquad
\mbox{for } t\in[0,1].
\end{equation}
\end{definition}

The following result is borrowed from \cite{AN}, Proposition~3 and
\cite{DU}, Proposition~4.2 and page 193 (we also refer to \cite{AN}, Section~5, for considerations on the indefinite Stratonovich integral). It
gives the link between Stratonovich and Young integrals with respect to $B$.


\begin{proposition} \label{prop:AN}
Let $u=\{u_t,  t\in[0,1]\}\in\mathbf{D}^{1,2}(|\mathcal{H}|)$, such that
%
\begin{equation}
\label{eq:intg-cdt-strato} \int_0^1 \!\!\int
_0^1 |\mathbf{D}_s
u_t| |t-s|^{2H-2}\,ds\,dt<\infty.
\end{equation}
Then:
\begin{longlist}[(ii)]
\item[(i)]  The Stratonovich integral $\sum_{k=1}^{d}\int_{0}^{1}
u_{s}^{k} \circ dB_{s}^{k}$
in the sense of Definition~\ref{def:strato-intg} exists, and we also have
%
\begin{equation}
\label{r:AN}
\sum_{k=1}^{d}\int
_{0}^{1} u_{s}^{k} \circ
dB_{s}^{k} =\bolds{\delta}(u)+\alpha_H \sum
_{k=1}^{d} \int_0^1\!\!
\int_0^1 \mathbf {D}_{s}^{k}
u_t |t-s|^{2H-2} \,ds\,dt.
\end{equation}

\item[(ii)]  Whenever $u\in\mathcal{C}^{\gamma}$ a.s. with $\gamma
>1/2$ and $H\in
(1/2,1)$, the Stratonovich integral $\sum_{k=1}^{d}\int_{0}^{1}
u_{s}^{k} \circ dB_{s}^{k}$ coincides with the Young integral $\sum_{k=1}^{d}\int_{0}^{1} u_{s}^{k}   \,dB_{s}^{k}$.
\end{longlist}
\end{proposition}

\begin{remark}
In the Brownian case (which corresponds to the limiting case $H\searrow1/2$),
one may wonder about the relation between our pathwise-type
Stratonovich integral and the Stratonovich integral of a square
integrable adapted process $u\in L_{a}^{2}$. The easiest way to carry
out this comparison might be to start with relation (\ref{r:AN}).
Indeed, on the right-hand side of this identity, the Skorohod integral
$\bolds{\delta}(u)$ coincides with It\^o's integral as long as $u\in
L_{a}^{2}$. As far as the terms $\alpha_H\int_0^1\!\int_0^1 \mathbf{D}_{s}^{k}
u_t |t-s|^{2H-2} \,ds\,dt$ is concerned, let us first mention that the
measure $2\alpha_H |t-s|^{2H-2} \,ds \,dt$ converges to the Lebesgue
measure on the diagonal $\{(s,t)\in[0,1]^{2};   s= t\}$ as $H\searrow
1/2$. We thus end up morally with a sum of terms of the form $\frac{1}2\int_0^1 \mathbf{D}_{t}^{k} u_t   \,dt$. The identification of this
term with
the bracket $\frac{1}2\langle u, W\rangle_{1}$ is then standard and is
detailed in \cite{Nu06}, Remark~2, page 175.
\end{remark}

The next Proposition allows us to interpret the stochastic
integral appearing in~(\ref{eq:sde-Vd}) as a Stratonovich-type
integral.

\begin{proposition}\label{prop:Xt-strato-B}
Let $X=\{X_t,  t\in[0,1]\}$ be the solution to
(\ref{eq:sde-Vd}), and assume Hypothesis~\ref{hyp:vector-fields} holds
true. Then $X\in\mathbf{D}^{1,2}(|\mathcal{H}|)$ and satisfies the equation
\[
\label{eq:Strato}
X_t=a + \sum_{k=1}^d
\int_0^t V_k(X_u)
\circ dB_u^{k},
\]
where the indefinite Stratonovich integral is defined by (\ref{eq:def-indefinite-strato-intg}), and can be decomposed as a Skorohod
integral plus a trace term as in (\ref{r:AN}).
\end{proposition}

\begin{pf}
According to Propositions \ref{prop:moments-sdes} and \ref{prop:AN}, we
just have to prove
that $X\in\mathbf{D}^{1,2}(|\mathcal{H}|)$ and satisfies
relation (\ref{eq:intg-cdt-strato}). We first focus on proving
the relation
\[
\mathbf{E} \bigl[\|X\|_{|\mathcal{H}|}^2 \bigr]+\mathbf{E} \bigl[\|
\mathbf{D}X\|_{|\mathcal{H}|\otimes
|\mathcal{H}|}^2 \bigr]<\infty.
\]
In order to see the first part of this inequality, invoke
relation (\ref{eq:bnd-moments-X-t}), and write
\begin{eqnarray*}
\mathbf{E} \bigl[\|X\|_{|\mathcal{H}|}^2 \bigr]&=&
\alpha_H\int_0^1\!\!\int
_0^1 \mathbf{E} \bigl[| X_r||X_s|
\bigr] |r-s|^{2H-2} \,dr\,ds
\nonumber
\\
&\leq& c \mathbf{E} \bigl[\|X\|_\infty^2 \bigr] \int
_0^1\!\!\int_0^1
|r-s|^{2H-2} \,dr\,ds<c_1. \label{Strato:cota1}
\end{eqnarray*}
Along the same lines and owing to
(\ref{eq:bnd-moments-D-n-X-t}), it is also readily checked that
$\mathbf{E}[\|\mathbf{D}X\|_{|\mathcal{H}|\otimes|\mathcal{H}|}^2]
<\infty$ and that
relation (\ref{eq:intg-cdt-strato}) holds true, which completes the
proof.
Note that due to Proposition~\ref{prop:AN}(ii) and Proposition~\ref{prop:cotes-hol-sol}, we obtain the other assertions.
\end{pf}

Finally, the following corollary is the key to the effective
decomposition we shall use in order to get our Gaussian lower
bound on $p_t$:

\begin{corollary}\label{cor:eq-X-strato-wrt-W}
Let the same assumptions as for Proposition~\ref{prop:Xt-strato-B} hold
true. For $0\le s \le t \le1$ and $\varphi\in|\mathcal{H}|$, we define
\[
K_t^{*}(\varphi)_{s} := \int
_{s}^{t} \varphi_r
\,\partial_r K(r,s) \,dr.
\]
Then the process $K^{*}_t(V_k(X))_\cdot\in\operatorname{Dom}(\delta)$ and satisfies
the equation
%
\begin{eqnarray}
X_t &=& a + \sum_{k=1}^d
\int_0^t \bigl[K_t^{*}
\bigl(V_k(X)\bigr) \bigr]_{s} \circ dW_s^{k}
\nonumber
\\[-8pt]
\label{eq:Strato-W}
\\[-8pt]
\nonumber
&=& a +\sum_{k=1}^d \int
_{0}^{t} \biggl(\int_{s}^{t}
\partial_{u} K(u,s) V_k(X_u) \,du \biggr)
\circ dW_{s}^k,
\end{eqnarray}
where the anticipative Stratonovich integrals with respect to $W$
can be decomposed as a Skorohod integral plus a trace term as follows:
%
\begin{eqnarray}
&& \sum_{k=1}^d \int
_0^t \bigl[K_t^{*}
\bigl(V_k(X)\bigr) \bigr]_{s} \circ dW_s^{k}
\nonumber
\\[-8pt]
\label{eq:strato-sko-with-W}\\[-8pt]
\nonumber
&& \qquad= \delta \bigl(K_t^{*}\bigl(V(X)\bigr) \bigr) + \sum
_{k=1}^d \int_0^t
D_{s}^{k} \bigl[K_t^{*}
\bigl(V_k(X)\bigr) \bigr]_{s} \,ds.
\end{eqnarray}
\end{corollary}

\begin{pf}
For notational sake, we give some details of the proof for $n=d=1$, the
easy adaptation to the multidimensional case being omitted. We also set
$V\equiv V_{1}$.
According to Proposition~\ref{prop:Xt-strato-B} and relation (\ref{r:AN}), we have $X_{t}=a+S_{t}+c_{H}   T_{t}$, with
\[
S_{t}= \bolds{\delta} \bigl(V(X) \mathbf{1}_{[0,t]} \bigr)
\quad\mbox{and} \quad T_{t} = \int_0^1\!\!\int
_0^1 \mathbf{D}_r \bigl(V(X)
\mathbf {1}_{[0,t]} \bigr)_{s} |r-s|^{2H-2} \,dr \,ds.
\]
Then owing to \cite{Nu06}, Proposition~5.2.2, we have $S_{t}=\delta
(K^{*}(V(X)\mathbf{1}_{[0,t]}))$. In\break addition, a direct and easy computation
shows that $K^{*}(V(X)\mathbf{1}_{[0,t]})=\break K_t^{*}(V_k(X))\mathbf
{1}_{[0,t]}$, so that
we have obtained
\[
S_{t} = \delta \bigl(K_t^{*}
\bigl(V_k(X)\bigr) \bigr),
\]
that is, the first term in (\ref{eq:strato-sko-with-W}).

Next, for a function $\varphi\dvtx [0,1]^{2}\to\mathbb{R}$ set
\[
\bigl[K^{*,\otimes2}\varphi \bigr]_{r_{1},r_{2}}= \int_{r_{1}}^{1}
\!\!\int_{r_{2}}^{1} \partial_{s_{1}}K(s_{1},r_{1})
\,\partial_{s_{2}}K(s_{2},r_{2})
\varphi_{s_{1}s_{2}} \,ds_{1} \,ds_{2}.
\]
Thanks to a slight extension of (\ref{eq:inner-pdt-H-smooth}), we get
\begin{eqnarray*}
T_{t} &=& \int_{0}^{1}
\bigl[K^{*,\otimes2} \bigl(\mathbf{D}V(X)\mathbf {1}_{[0,t]} \bigr)
\bigr]_{s,s} \,ds = \int_{0}^{1}
D_{s} \bigl[K^{*}\bigl(V(X)\mathbf{1}_{[0,t]}
\bigr) \bigr]_{s} \,ds \\
&=&  \int_{0}^{t}
D_{s} \bigl[K_{t}^{*}\bigl(V(X)\bigr)
\bigr]_{s} \,ds,
\end{eqnarray*}
where the second relation is due to Proposition~\ref{prop:relation-D-DW}, and the third one stems from the fact that
$K^{*}(V(X)\mathbf{1}_{[0,t]})=K_t^{*}(V_k(X))\mathbf{1}_{[0,t]}$.
Gathering the
expressions we have obtained for the two terms $S_{t}$ and $T_{t}$, the
proof of our claim (\ref{eq:strato-sko-with-W}) is now complete.
\end{pf}

\subsection{Discretization procedure}
We now proceed to the decomposition of $F_n:=X_t$ as
announced in (\ref{eq:dcp-Fj-I-Rd}), starting from the
expression of $F_j$ for $j=0,\ldots,n$. Indeed, according to expression
(\ref{eq:Strato-W}), a natural approximation sequence for $X_t$ based
on a partition
$0=t_0<\cdots<t_n=t$ of $[0,t]$ is the following:
%
\begin{equation}
\label{eq:recurs-Fi-Fi-1}
F_{i}= F_{i-1} + I_{i} +
R_{i},
\end{equation}
where, introducing the additional notation
%
\begin{equation}
\label{eq:notation-eta-i-gk-i}
\eta_{i}(u):=\inf(u,t_{i})\quad \mbox{and}\quad
g_{i,s}^k :=\int_{s}^{t}
\partial_{u} K(u,s) V_k(X_{\eta_{i}(u)}) \,du,
\end{equation}
we set (note that $g_{i-1,s}^k\in\mathcal{F}_{t_{i-1}}$)
%
\begin{eqnarray}
F_{i-1}&:=&\sum_{k=1}^d \int
_{0}^{t_{i-1}} g_{i-1,s}^k \circ
dW_{s}^k,
\nonumber
\\[-8pt]
\label{eq:exp-I-i-cdt-gauss}
\\[-8pt]
\nonumber
I_{i} &:=& \sum_{k=1}^d\int
_{t_{i-1}}^{t_{i}} g_{i-1,s}^k \circ
dW_{s}^k= \sum_{k=1}^d
V_k(X_{t_{i-1}}) \int_{t_{i-1}}^{t_{i}}
K(t,s) \,dW_{s}^k,
\end{eqnarray}
where the last integral above is simply a Wiener integral with respect
to $W$. We also introduce a family of random variables $R_{i}$ defined by
%
\begin{equation}
\label{eq:def-Ri}
R_{i}:= \sum_{k=1}^d
\int_{t_{i-1}}^{t_{i}} Q^k_s
\circ dW_{s}^k,
\end{equation}
where $Q$ is the process defined by
%
\begin{equation}
\label{eq:def-Qk} Q_s^k: = \int_s^t
\partial_u K(u,s) \bigl[V_k(X_{\eta_i(u)})-V_k(X_{t_{i-1}})
\bigr]\,du.
\end{equation}
Observe that if $V$ is elliptic and bounded, it is clear from
expression (\ref{eq:exp-I-i-cdt-gauss}) that
$\sum_{i} \operatorname{Cov}_{t_{i-1}}(I_{i})\asymp t^{2H} Id_m$ up
to a
constant, independently of the particular values of the $t_{i}$'s. We
shall see, however, how to choose those values in Condition \ref{cdt:points-partition}.

Finally we introduce some random variables
$\Phi_{M}(\mathrm{N}_{\gamma,p}^i(B))$ for $i=1, \ldots, n$ which
allow us to control the supremum norm of the solution of
equation (\ref{eq:sde-Vd}) and of their stochastic derivatives.
This argument needs to be added in the methodology of \cite{Ba,Ko}, and
therefore we have to tailor the arguments therein to our situation.
The localization random variables are based on the family of functionals
$\mathrm{N}_{\gamma,p}^i(B)$ defined by
\[
\mathrm{N}_{\gamma,p}^i(B)=\int_{t_{i-1}}^{t_i}
\!\int_{t_{i-1}}^{t_i} \frac{|B_v-B_u|^{2p}}{|v-u|^{2\gamma p+2}} \,du\,dv,
\]
which can be compared to H\"older-type norms and have the
advantage that they can be differentiated with respect to $B$.
In fact, we can see the aim of introducing this functional in
the following proposition, which is direct consequence of the
Garsia--Rodemich--Rumsey's lemma; see, for example, \cite{Ga}.

\begin{proposition} \label{prop:rel-normB-N}
Let $H>\frac{1}2$ and $p$ such that
$0<\gamma<H-\frac{1}{2p}$. Then we have
$\|B\|_{t_{i-1},t_i,\gamma}\leq c_{\gamma,p}
[\mathrm{N}_{\gamma,p}^i(B)]^{1/2p}$.
\end{proposition}

The next step is to study the conditional densities of the
approximation sequence~$F_i$.
To this end, one has to control various terms for which the
localization technique of Malliavin Calculus turns out to be useful.
Specifically, recall that we have introduced families of functions
$\Phi
_{M},\Phi_{M,\epsilon}$ given by expression (\ref{eq:def-Phi}).
In the
sequel we localize our expectations using
functionals of the type $\Phi_{M}(\mathrm{N}_{\gamma,p}^i(B))$ and
$\Phi
_{c_i,\epsilon}(\sum_{j=1}^d\int_{t_{i-1}}^{t_i}|\mathrm
{D}^j_rR_i|^2\,dr)$ for some constants $c_{i},\epsilon$ of the form
%
\begin{equation}
\label{eq:def-ci-epsiloni}
c_i:=\frac{\lambda}{4} \int_{t_{i-1}}^{t_i}K^{2}(t,s)
\,ds>0 \quad \mbox{and}\quad\epsilon_i:=\frac{c_i}2>0.
\end{equation}
Furthermore, in order to ease notation, notice that we will simply write
%
\begin{equation}
\label{eq:star}
\Phi_{M}\equiv\Phi_{M}\bigl(
\mathrm{N}_{\gamma,p}^i(B)\bigr)\quad \mbox{and}\quad \Phi_{c_i,\epsilon_i}
\equiv\Phi_{c_i,\epsilon_i} \Biggl(\sum_{j=1}^d
\int_{t_{i-1}}^{t_i}\bigl|\mathrm{D}^j_rR_i\bigr|^2
\,dr \Biggr).
\end{equation}

With this additional notation in hand, we can proceed to the first step
of our approximation scheme: since $F_i$ is $\mathcal{F}_{t_{i-1}}$
conditionally
nondegenerate and the localizations $\Phi_{M}$ and $
\Phi_{c_i,\epsilon_i}\in D^\infty$, we can write
\[
\mathbf{E}_{t_{i-1}} \bigl[\delta_x(F_i)\bigr] =
\mathbf{E}_{t_{i-1}} \bigl[\delta_x(F_i)
\Phi_{M} \Phi_{c_i,\epsilon_i}\bigr]+\mathbf{E}_{t_{i-1}} \bigl[
\delta_x(F_i) (1-\Phi_{M}
\Phi_{c_i,\epsilon_i}) \bigr],
\]
and due to the nonnegativity of the second term, we have
\[
\mathbf{E}_{t_{i-1}} \bigl[\delta_x(F_i)\bigr]
\geq\mathbf{E}_{t_{i-1}} \bigl[\delta_x(F_i)
\Phi_{M}\Phi_{c_i,\epsilon_i}\bigr].
\]
Recalling that $F_{i}=F_{i-1}+I_{i}+R_{i}$, we then obtain the
following decomposition:
%
\begin{equation}
\label{eq:dcp-delta-Fi}
\mathbf{E}_{t_{i-1}} \bigl[\delta_x(F_i)
\Phi_M\Phi_{c_i,\epsilon
_i} \bigr]= J_{1,i} +
J_{2,i} + J_{3,i},
\end{equation}
where
%
\begin{equation}
\label{eq:def-J1-J2} \quad J_{1,i}= \mathbf{E}_{t_{i-1}} \bigl[
\delta_x(F_{i-1}+I_{i})\bigr],\qquad
J_{2,i}= \mathbf{E}_{t_{i-1}} \bigl[\delta_x(F_{i-1}+I_{i})
(\Phi_{M}\Phi_{c_i,\epsilon_i}-1)\bigr]\hspace*{-16pt}
\end{equation}
and
%
\begin{equation}
\label{eq:def-J3}
J_{3,i}= \sum_{j=1}^m
\mathbf{E}_{t_{i-1}} \biggl[ \Phi_{M} \Phi_{c_i,\epsilon_i} \int
_{0}^{1} \partial_{x_j}
\delta_x(F_{i-1}+I_{i}+\rho R_i)
R_i^j \,d\rho \biggr].
\end{equation}
Our aim is now to prove that in this decomposition, $J_{1,i}$
should yield the main contribution, while $J_{2,i}$ is small
because of the quantity $(\Phi_{M}\Phi_{c_i,\epsilon_i}-1)$
whenever $M$ and $n$ are large enough, and $J_{3,i}$ is small
due to the presence of the difference between $X_{t_i}-X_{t_{i-1}}$ in
$R_{i}$. We shall implement this strategy in the next subsections.

\subsection{Upper and lower bounds on $J_{1,i}$}
The main information which will be used about $J_{1,i}$ is the
following:

\begin{proposition}\label{prop:exp-J1-i-gauss}
Let $J_{1,i}$ be defined by (\ref{eq:def-J1-J2}). Then under
Hypothesis \ref{hyp:vector-fields} we have
%
\begin{equation}
\label{eq:exp-J1d}
\quad J_{1,i}=\mathbf{E}_{t_{i-1}} \bigl[
\delta_x(F_{i-1}+I_{i})\bigr] =
\frac{\exp (-({1}/2)
(x-F_{i-1})^{*} \Sigma_{i-1}^{-1}(x-F_{i-1}) )}{ (
2\pi )^{m/2} |\Sigma_{i-1}|^{1/2}},\hspace*{-8pt}
\end{equation}
where $\Sigma_{i-1}$ is a deterministic (conditionally to
$\mathcal{F}_{t_{i-1}}$) matrix such that
\[
\lambda \biggl(\int_{t_{i-1}}^{t_{i}}
K^{2}(t,u) \,du \biggr)\mbox{Id}_{m} \le
\Sigma_{i-1} \le\Lambda \biggl(\int_{t_{i-1}}^{t_{i}}
K^{2}(t,u) \,du \biggr) \mbox{Id}_{m},
\]
and where the two strictly positive constants $\lambda,\Lambda$
satisfy (\ref{eq:elliptic-condition-V}).
\end{proposition}

\begin{pf}
The fact that $I_{i-1}$ is conditionally Gaussian is clear from
expression~(\ref{eq:exp-I-i-cdt-gauss}), and this immediately
yields our claim (\ref{eq:exp-J1d}). Furthermore,
\begin{eqnarray*}
\Sigma_{i-1}&:=& \operatorname{Cov}_{t_{i-1}}(I_i)=
\mathbf {E}_{t_{i-1}}\bigl[I_i I_i^*\bigr]
\\[-2pt]
&=&\mathbf{E}_{t_{i-1}} \Biggl[ \Biggl(\sum_{k=1}^d
V_k(X_{t_{i-1}})\int_{t_{i-1}}^{t_i}
K(t,u) \,dW_u^k \Biggr)\\[-2pt]
&&\hspace*{6pt}\qquad{}\times \Biggl(\sum
_{l=1}^d V_l^*(X_{t_{i-1}})\int
_{t_{i-1}}^{t_i} K(t,u) \,dW_u^l
\Biggr) \Biggr]
\\[-2pt]
&=& \sum_{k=1}^d V_k(X_{t_{i-1}})V_k^*(X_{t_{i-1}})
\int_{t_{i-1}}^{t_{i}} K^{2}(t,u) \,du,
\end{eqnarray*}
which completes the proof of our second claim, thanks to
Hypothesis \ref{hyp:vector-fields}.
\end{pf}

The previous proposition induces a natural choice for the
partition $(t_i)$ in terms of the kernel $K$:

\begin{condition}\label{cdt:points-partition}
We choose the partition $0=t_0<\cdots<t_n=t$ of $[0,t]$ such
that we have $\int_{t_{i-1}}^{t_{i}} K^{2}(t,u)   \,du=\frac
{t^{2H}}{n}=:\sigma_n^2$
for all $i=1,\ldots,n$.
\end{condition}

With this choice in hand, let us note the following properties
for further use:

\begin{lemma}\label{lem:prop-partition}
Let $t_0,\ldots,t_n$ be the partition of $[0,t]$ defined by
Condition~\ref{cdt:points-partition}. Then:
\begin{longlist}[(iii)]
\item[(i)]  The partition is constructed in a unique way.
\item[(ii)] We have $0\le t_i-t_{i-1} \le c_{H}   n^{-1/(2H)}$
for all $i=1,\ldots,n$.
\item[(iii)]  The parameters $c_{i}$ defined at (\ref{eq:def-ci-epsiloni}) are all equal to $\frac{\lambda t^{2H}}{4n}$.
\end{longlist}
\end{lemma}

\begin{pf}
Our\vspace*{1pt} first claim stems from the fact that $\int_{0}^{t}
K^{2}(t,u)   \,du=t^{2H}$ and $v\mapsto\int_{v}^{\tau}
K^{2}(t,u)   \,du$ is a strictly decreasing function for all
$0\le v\le\tau\le t$.

In order to prove our item (ii), recall expression
(\ref{eq:def-kernel-K}), from which we easily deduce the bound
%
\begin{equation}
\label{eq:low-bnd-K}
K(t,s)\ge c_H (t-s)^{H-1/2}.
\end{equation}
Consider now a fixed point $\tau\in(0,t]$ and $0\le v\equiv
v_{\tau} < \tau\le t$ such that $\int_{v}^{\tau} K^{2}(t,u)
\,du=\frac{t^{2H}}{n}$. Thanks to bound (\ref{eq:low-bnd-K})
we have $v_{\tau}\ge w_{\tau}$ where $w_{\tau}\equiv w$ is
defined by
\[
c_H \int_{w}^{\tau}
(t-u)^{2H-1} \,du=\frac{t^{2H}}{n}\quad \Longleftrightarrow \quad c_H
\bigl[(t-w)^{2H} - (t-\tau)^{2H} \bigr]=\frac{t^{2H}}{n}.
\]
In addition, since $2H>1$, we have $(t-w)^{2H} - (t-\tau)^{2H}\ge
(\tau-w)^{2H}$ for $w<\tau<t$, which means that $w_\tau\ge
x_\tau$ where $x_\tau$ is defined by the equation
$(\tau-x)^{2H}= \frac{c_H t^{2H}}{n}$. The latter equation can
be solved explicitly as $x_{\tau}=\tau- \frac{c_H
t}{n^{1/(2H)}}$, and summarizing our last considerations we end
up with the relation
\[
\tau- v_{\tau} \le\frac{c_H t}{n^{1/(2H)}},
\]
which easily yields our assertion (ii).
The proof of (iii) is straightforward.
\end{pf}

Now we state the following corollary to Proposition~\ref{prop:exp-J1-i-gauss}, whose immediate proof is left to the
reader:

\begin{corollary}\label{cor:low-bnd-J1}
Let $J_{1,i}$ be defined by (\ref{eq:def-J1-J2}). Then under
Hypothesis \ref{hyp:vector-fields} and
Condition~\ref{cdt:points-partition} we have for $\sigma_n^{2}=\frac
{t^{2H}}{n}$
%
\begin{equation}
\label{eq:low-bnd-J1i}
J_{1,i} \ge \frac{1}{(2\pi)^{m/2}(\Lambda\sigma_n^2)^{m/2}} \exp \biggl( -
\frac{|x-F_{i-1}|^{2}}{2\lambda  \sigma_n^{2}} \biggr).
\end{equation}
\end{corollary}

Summarizing the considerations of this section, we have obtained
that the main contribution to $\mathbf{E}_{t_{i-1}} [\delta_x(F_i)]$,
$J_{1,i}$, is of the order given by
(\ref{eq:low-bnd-J1i}). Most of our work is now devoted to prove
that the contributions of $J_{2,i}$ and $J_{3,i}$ are
smaller than a fraction of (\ref{eq:low-bnd-J1i}) if $M,n$ are
conveniently chosen.

\subsection{Upper bounds for $J_{2,i}$}

We start the control of $J_{2,i}$ by stating a bound in terms of the
localization we have chosen:

\begin{proposition} \label{prop:bound-J2}
Let $J_{2,i}$ be the quantity defined by (\ref{eq:def-J1-J2}).
Then there exists positive constants $c_{\lambda,\Lambda}$, $k_1, k_2$
and $p_1$ independent of $n$ such that
%
\begin{eqnarray}
| J_{2,i}| &\le & c_{\lambda,\Lambda} \bigl({\sigma_n^2}
\bigr)^{-k_2} L_{n,i}^{\gamma,p}(k_1,p_1)\nonumber\\
\eqntext{\mbox{where } L_{n,i}^{\gamma,p}(k_1,p_1)
\equiv\|1-\Phi_M\Phi_{c_i,\epsilon_i}\| _{k_1,p_1,t_{i-1}},}
\end{eqnarray}
with $\sigma_n^2=\frac{t^{2H}}{n}$, and\vspace*{1pt} where we recall that
the norms $\|\cdot\|_{k,p,t}$ have been introduced at equation~(\ref{eq:def-conditional-sobolev}) and the random variables $\Phi_M,\Phi
_{c_i,\epsilon_i}$ at equation (\ref{eq:star}).
\end{proposition}

\begin{pf}
Our strategy hinges on the conditional integration by parts
formula we have introduced in Proposition~\ref{prop:int-parts-cond}, which gives for some constants $k_i,p_i$,
$i=1,\ldots,4$,
%
\begin{eqnarray}
\quad\hspace*{8pt} |J_{2,i}| &=&  \bigl\llvert \mathbf{E}_{t_{i-1}}
\bigl[\mathbf{1}_{\{
F_{i-1}+I_{i} > x\}} H^{t_{i-1}}_{(1,\ldots,m)}(I_i,
1-\Phi_{M}\Phi_{c_i,\epsilon_i}) \bigr]\bigr\rrvert
\nonumber
\\[-8pt]
\label{eq:bnd-H1d-I-Theta}
\\[-8pt]
\nonumber
&\le& c_{1,q} \bigl\|\det(\Gamma_{I_i,t_{i-1}})^{-1}
\bigr\|_{p_3,t_{i-1}}^{k_3} \|I_i\|_{k_2,p_2,t_{i-1}}^{k_4}
\| 1-\Phi_{M}\Phi_{c_i,\epsilon_i}\|_{k_1,p_1,t_{i-1}}.
\end{eqnarray}
Here, we have used that $\mathbf{1}_{\{F_{i-1}+I_{i} > x\}}\le1$.

In order to bound the right-hand side of
(\ref{eq:bnd-H1d-I-Theta}) we start by computing the Malliavin
derivatives of $I_i$.
Recall that due to (\ref{eq:exp-I-i-cdt-gauss}), we have for
$j=1, \ldots,  d$, $\alpha>1$ and
$r,r_{1},\ldots,r_{\alpha}>t_{i-1}$ that
\[
\label{eq:expression-deriv-I} \mathrm{D}_r^{j} I_i=V_j(X_{t_{i-1}})K(t,r)
\mathbf{1}_{[t_{i-1},
t_{i}]}(r)\quad \mbox{and}\quad \mathrm{D}^{\alpha}_{r_1\ldots r_\alpha}
I_i=0.
\]
As far as $\Gamma_{I_i,t_{i-1}}$ is concerned, it is a
conditionally deterministic quantity such that for
$i, j=1, \ldots, d$, we can write
\begin{eqnarray*}
\Gamma_{I_i,t_{i-1}}&=&\sum_{j=1}^{d}
\bigl\langle\mathrm{D}^{j} I_i, \mathrm{D}^{j}
I_i^{*} \bigr\rangle_{L^2([t_{i-1},t_i])}
\\
&=&\sum_{j=1}^{d}V_{j}(X_{t_{i-1}})V_{j}^{*}(X_{t_{i-1}})
\int_{t_{i-1}}^{t_i} K^2(t,s)\,ds =
\sigma_n^2 V(X_{t_{i-1}})V^{*}(X_{t_{i-1}}).
\end{eqnarray*}
Using the ellipticity condition of Hypothesis
\ref{hyp:vector-fields}(2) for $V$, we thus obtain that
\[
0\leq \Gamma_{I_i,t_{i-1}}^{-1}\leq\frac{1}{\lambda\sigma_n^2}Id_m.
\]
Therefore $\|I_i\|_{k_2,p_2,t_{i-1}}^{k_4}\le C ({\sigma
_n^2\Lambda} )^{{k_4}/2}$ and
\[
\bigl\|\det(\Gamma_{I_i,t_{i-1}})^{-1}\bigr\|_{p_3,t_{i-1}}^{k_3}
\le \biggl(\frac
{1}{\lambda\sigma_n^2} \biggr)^{mk_3}.
\]
Substituting these inequalities in (\ref{eq:bnd-H1d-I-Theta}), our
proof is now finished.
\end{pf}

From the above Proposition~\ref{prop:bound-J2}, we see that in order to
get a
convenient bound for $J_{2,i}$ we need to study the random variable
$\|1-\Phi_{M}\Phi_{c_i,\epsilon_i}\|_{k_1,p_1,t_{i-1}}$. A suitable
information for us will be the following bound:

\begin{proposition} \label{prop:bound-norm-phi}
Assume Condition~\ref{cdt:points-partition} and consider any $\gamma
\in(\frac{1}2,H)$ and $k_{1},p_{1}\ge1$. Let
$L_{n,i}^{\gamma,p}(k_1,p_1)=\|1-\Phi_{M}\Phi_{c_i,\epsilon_i}\|
_{k_1,p_1,t_{i-1}}$
be the random variable defined at Proposition~\ref{prop:bound-J2}.
Then for any $p\ge\frac{k_1}2, \gamma>0$ [recall
that $\Phi_M\equiv\Phi_{M}(\mathrm{N}_{\gamma,p}^i(B))$] such that
$2p(H-\gamma)-2>k_1H$
the following holds true: For any $\eta>0$ there exists
$c_{p,k_1,p_1,\gamma,H,M,\eta}>0$ such that
%
\begin{equation}
\label{eq:bnd-Lni}
\mathbf{E}\bigl[ L_{n,i}^{\gamma,p}(k_1,p_1)
\bigr] \leq c_{p,k_1,p_1,\gamma,H,M,\eta} n^{-\eta}.
\end{equation}
\end{proposition}

\begin{pf}
Let us first highlight what the parameters involved in the proof are:
recall that $c_i$ and $\epsilon_i$ were defined in
(\ref{eq:def-ci-epsiloni}). And although not explicitly written,
$\Phi_{M}$ depends on $\gamma$ and $p$. From now on, and through the
proof we fix the values of $\gamma$, $H$, $k_1$, $p_1$, $n$ and $p$
satisfying the inequalities in the statement of the proposition.

As a preliminary step, we also observe that, due to the H\"older
inequality, it is enough to find a proper bound for
$\|1-\Phi_{M}\|_{k_1,p_1,t_{i-1}}$ and $\|\Phi_{M}(1-\Phi
_{c_i,\epsilon
_i})\|_{k_1,p_1,t_{i-1}}$ separately. We\vspace*{1pt} first handle the term $\|
1-\Phi
_{M}\|_{k_1,p_1,t_{i-1}}$.

Now we will obtain a general estimate to be used in the proof. By
Chebyshev's inequality, for any $k_2\ge1$ and $\frac{1}2<\gamma<H$,
%
\begin{equation}
\label{eq:Markov-NiB} \mathbf{E} \bigl[|1-\Phi_{M}|^2 \bigr]
\leq \mathbf{P}\bigl(\mathrm{N}_{\gamma,p}^i(B)>M-1\bigr) \leq
\frac{\mathbf{E} [|\mathrm{N}_{\gamma,p}^i(B)|^{k_2}
]}{(M-1)^{k_2}}.
\end{equation}
We now find an upper bound for $\mathbf{E}[|\mathrm{N}_{\gamma
,p}^i(B)|^{k_2}]$. A simple
application of Jensen's inequality yields
%
\begin{eqnarray}
\mathbf{E} \bigl[\bigl|\mathrm{N}_{\gamma,p}^i(B)\bigr|^{k_2}
\bigr] &=& \mathbf{E} \biggl[ \biggl(\int_{t_{i-1}}^{t_i}
\!\int_{t_{i-1}}^{t_i}\frac
{|B_v-B_u|^{2p}}{|v-u|^{2p\gamma+2}}\,du\,dv
\biggr)^{k_2} \biggr]
\nonumber
\\
\label{eq:bnd-NiB}
&\leq & c |t_{i}-t_{i-1}|^{2(k_2-1)} \biggl(\int
_{t_{i-1}}^{t_i}\!\int_{t_{i-1}}^{t_i}
\frac{\mathbf{E} [|B_{v}-B_{u}|^{2pk_2}
 ]}{|v-u|^{(2p\gamma+2)k_2}}\,du\,dv \biggr)\\
 & \leq &  c_{k_2,p,\gamma,H} |t_i-t_{i-1}|^{2k_2p(H-\gamma)}.\nonumber
\end{eqnarray}
We remark that all above integrals and expectations are finite due to the
condition $2p(H-\gamma)-2>k_1H$. Furthermore, the quantity $
|t_i-t_{i-1}|^{2k_2p(H-\gamma)}$ can be made as small as we wish by
taking $k_{2},p$ and $n$ large enough. We will play on these parameters
later on.

Let us start the estimation for the high-order
derivatives of $1-\Phi_{M}$. For this, we first notice that, for any
$\mathbf{r}$ of length greater or equal to 1 and any $\mathbf{i}$, we
have $\mathbf{D}^{\mathbf{i}}_{\mathbf{r}}(1-\Phi_{M})=-\mathbf
{D}^{\mathbf{i}}_{\mathbf{r}}\Phi_{M}$, so that we shall bound $\mathbf{D}^{\mathbf{i}}_{\mathbf{r}}\Phi_{M}$ in the sequel.
Next we need to define the
set of multi-indices $\mathcal{A}_n=\{(l_1,\ldots,l_n);l_i\in\{0,\ldots,n\},
l_1+\cdots+l_n=n\}$. In fact, one can easily check that there exist
(explicit) random variables $\mu_{p,l,\gamma,H}^{\mathbf{i}}(\mathbf
{r})$, defined for $l\le n\le k_1$, $\mathbf{r}=(r_1,\ldots,r_n)$ with
$r_1\le\cdots\le r_n$ and $\mathbf{i}=(i_1,\ldots,i_n)\in\{1,\ldots,d\}^n$,
such that the following inequality holds for a positive constant
$C_{p,l,\gamma,H}(\mathbf{i},\mathbf{r})$:
%
\begin{equation}
\label{eq:bnd-D-Phi-M-1}
\bigl|\mathbf{D}^{\mathbf{i}}_{\mathbf{r}}\Phi_{M}\bigr|\le
\sum_{l=1}^n\bigl|\partial_z^l
\Phi _{M}\bigl(\mathrm{N}_{\gamma,p}^i(B)\bigr)\bigr|\bigl|
\mu_{p,l,\gamma,H}^{\mathbf
{i}}(\mathbf{r})\bigr|,
\end{equation}
and where the random variables $\mu_{p,l,\gamma,H}^{\mathbf
{i}}(\mathbf{r})$ satisfy
%
\begin{eqnarray}
&& \bigl|\mu_{p,l,\gamma,H}^{\mathbf{i}}(\mathbf{r})\bigr|\le C_{p,l,\gamma,H} \prod
_{\mathbf{l}\in\mathcal{A}_l; j=1}^l\mu_{p,l_j,\gamma,H}\nonumber\\
\eqntext{\ds\mbox{with }
\mu_{p,l,\gamma,H}= \int_{t_{i-1}}^{t_i}\! \int
_{t_{i-1}}^{t_i} \frac{|B_\xi-B_\eta|^{2p-l}}{|\xi-\eta|^{2\gamma p+2}}\,d\xi \,d\eta.}
\end{eqnarray}
Note that all the integrals above are well defined due to the
restrictions $2p\ge k_1$ and $2p(H-\gamma)-2>k_1H$.

Next, we estimate the moments of $\mu_{p,l,\gamma, H}^{\mathbf
{i}}(\mathbf{r})$ as follows. For any $\kappa\in\mathbb{N}$, we have
%
\begin{eqnarray}
&& \mathbf{E}\bigl[|\mu_{p,l,\gamma, H}|^{\kappa}\bigr]\nonumber\\
&&\label{eq:74}\qquad\leq
C_{p,l,\gamma, H}(t_i-t_{i-1})^{2(\kappa-1)}\int
_{t_{i-1}}^{t_i}\!\int_{t_{i-1}}^{t_i}
\frac{\mathbf{E} [|B_\xi-B_\eta|^{(2p-l)\kappa
} ]}{|\xi-\eta
|^{(2\gamma
p+2)\kappa}}\,d\xi \,d\eta
\\
\nonumber
&&\qquad\leq  c_{p,l,\kappa,\gamma,H} |t_i-t_{i-1}|^{2p\kappa(H-\gamma
)-l\kappa
H}.
\end{eqnarray}
Therefore $\|\mu_{p,l,\gamma, H}\|_{\kappa}\le c_{p,l,\kappa,\gamma
,H}
|t_i-t_{i-1}|^{2p(H-\gamma)-l
H}$. Note again that here, we have used the hypothesis $2p(H-\gamma)-2>k_1H$.

Let us now turn to the estimation of $\mathbf{D}^{n}_{\mathbf{r}}\Phi_{M}$.
Starting from relation (\ref{eq:bnd-D-Phi-M-1}), we get for $n\ge1$,
\begin{eqnarray*}
&& \bigl\|\mathbf{D}^{n}_{\mathbf{r}}\Phi_{M}
\bigr\|^2_{|\mathcal
{H}([t_{i-1},t_i])|^{\otimes n}} \\
&&\qquad\le  \sum_{l,m=1}^{n}
\prod_{\mathbf{l}\in\mathcal{A}_l; j=1}^{l} \prod
_{\mathbf{l}\in
\mathcal
{A}_m; k=1}^{m} |\mu_{p,l_j,\gamma,H}| |
\mu_{p,m_k,\gamma,H}|
\\
&&\hspace*{91pt}\qquad\qquad{}\times \bigl|\partial_z^l\Phi_{M}\bigl(
\mathrm{N}_{\gamma,p}^i(B)\bigr)\bigr| \bigl|\partial_z^m
\Phi_{M}\bigl(\mathrm{N}_{\gamma,p}^i(B)\bigr)\bigr| \\
&& \qquad\quad{}\times\int
_{[t_{i-1},t_i]^{2n}}\prod_{i=1}^n|r_i-s_i|^{2(H-1)}
\,dr_i\,ds_i.
\end{eqnarray*}
Finally, plugging our previous inequalities (\ref{eq:bnd-NiB}) and
(\ref{eq:74}) and resorting to H\"older's inequality with $\mathbf
{q}=(q_1,\ldots,q_{l+m+1})$ where $q_1^{-1}+\cdots+q_{l+m+1}^{-1}=1$, we have
for $k_1\ge1$,
\begin{eqnarray*}
&& \mathbf{E} \bigl[\bigl\|\mathbf{D}^{k_1}\Phi_{M}
\bigr\|^{p_1}_{L^2([t_{i-1},t_i]^n)} \bigr]\\
 &&\qquad\leq   c_{p,k_1,\kappa,\gamma,H} \|
\Phi_M\|^{p_1}_{{k_1},\infty} \mathbf{P}^{1/q_1}
\bigl(\mathrm{N}_{\gamma,p}^i(B)>M-1\bigr)
\\
&&\qquad\quad{}\times \sum_{l,m=1}^{k_1} \prod
_{\mathbf{l}\in\mathcal{A}_l; j=1}^{l} \prod_{\mathbf{l}\in
\mathcal
{A}_m; k=1}^{m}
\|\mu_{p,l_j,\gamma,H}\|^{p_1}_{q_{j+1}} \|\mu_{p,m_k,\gamma,H}
\|^{p_1}_{q_{l+1+k}}\\
&&\qquad\quad{}\times \Biggl(\int_{[t_{i-1},t_i]^{2n}} \prod
_{i=1}^{k_1}|r_i-s_i|^{2(H-1)}
\,dr_i\,ds_i \Biggr)^{p_1}
\\
&&\qquad\leq   c_{p,k_1,p_1,\mathbf{q},\gamma,H,k_2} \|\Phi_M\|^{p_1}_{n,\infty}
|t_i-t_{i-1}|^{(k_2q_1^{-1}+4)pp_1(H-\gamma)},
\end{eqnarray*}
where we have set $\|\Phi_M\|_{n,\infty}:=\sum_{l=0}^n\|\partial
_z^l\Phi
_M\|_\infty$.
Therefore the result follows from (\ref{eq:bnd-NiB}) and the above
inequality by noting that $|t_i-t_{i-1}|\le c_{H}   n^{-1/(2H)}$ and
taking $k_2$ big enough. We remark that this result also
gives that\break $\|\Phi_{M}\|_{k_1,p_1,t_{i-1}}\le c_{p,k_1,p_1,\mathbf
{q},\gamma,H}$.

The calculation for $\|\Phi_{M}(1-\Phi_{c_i,\epsilon_i})\|
_{k_1,p_1,t_{i-1}}$ is similar, recalling that the norm of the
Malliavin derivatives of $\Phi_{M}$ are bounded, and noting that
instead of applying the operator $\mathbf{D}^{k_1}_{\mathbf{r}}$, it
is better
to use directly the derivative operator $\mathrm{D}^{k_1}_{\mathbf{r}}$
with Lemma~\ref{lema:cotaQ}. We skip details for sake of conciseness.
Observe, however, that in this case, the derivatives of $1-\Phi
_{c_i,\epsilon_i}$ blow up as $c_{i},\varepsilon_{i}$ get small.
Still, one
remarks that the final proof is based on the fact that for any
$k_{6}>0$, Chebyshev's inequality and the proof of Lemma~\ref{lem:app1}
(postponed to the \hyperref[sec:app1]{Appendix}) imply that
\begin{eqnarray*}
&&\mathbf{P} \Biggl(\sum_{j=1}^d\int
_{t_{i-1}}^{t_i}\bigl|\mathrm {D}^j_rR_i\bigr|^2
\,dr>\frac{\lambda}8\int_{t_{i-1}}^{t_i}K^{2}(t,s)
\,ds \Biggr)
\\
&&\qquad\le \biggl(\frac{\lambda}8\int_{t_{i-1}}^{t_i}K^{2}(t,s)
\,ds \biggr)^{-k_{6}} \mathbf{E} \Biggl[ \Biggl(\sum
_{j=1}^d\int_{t_{i-1}}^{t_i}\bigl|
\mathrm{D}^j_rR_i\bigr|^2\,dr
\Biggr)^{k_{6}} \Biggr]
\\
&&\qquad\le c \bigl( \lambda\sigma_n^2 \bigr)^{-k_{6}}(t_i-t_{i-1})^{(2\gamma+ 1) k_{6}}
\le cn^{-({\gamma}/{H})k_{6}}.
\end{eqnarray*}
Here we have used the result in Lemma~\ref{lem:prop-partition}(ii) and
Condition \ref{cdt:points-partition}.
\end{pf}
%


\subsection{Upper bounds for $J_{3,i}$}

We now turn to the main technical issue in our computations,
namely the bound on $J_{3,i}$. Our aim is thus to prove the
following proposition:

\begin{proposition}\label{prop:bnd-J3i}
Let $J_{3,i}$ be the quantity defined by (\ref{eq:def-J3}). Then
there exist $c>0$ and $k>0$ such that for any $H-\frac{1}2<\gamma<H$,
%
\begin{equation}
\label{eq:bnd-J3i}
|J_{3,i}| \le \frac{c_{M,V,m} (t_i-t_{i-1})^{\gamma}}{(\sigma_n^2)^{m/2} } \le
\frac{c_{M,V,m}}{n^{\gamma/2H}   (\sigma_n^2)^{m/2} }.
\end{equation}
\end{proposition}

\begin{pf}
We start from expression (\ref{eq:def-J3}) and normalize $I_{i}+\rho
R_i$ in the following way: we just set $I_{i}+\rho R_i=\sigma_{n}
\mathcal{U}
_{i}$, where $\mathcal{U}_{i}:=\sigma_{n}^{-1} (I_{i}+\rho R_i)$. We
thus have
\[
J_{3,i}= \sum_{j=1}^m
\mathbf{E}_{t_{i-1}} \biggl[ \Phi_{M} \Phi_{c_i,\epsilon_i} \int
_{0}^{1} \partial_{x_j}
\delta_x(F_{i-1}+\sigma_{n}
\mathcal{U}_{i}) R_i^j \,d\rho \biggr].
\]
Along the same lines as in (\ref{eq:bnd-H1d-I-Theta}), the integration
by parts formula (\ref{Hub}) now yields
\[
J_{3,i} = \sigma_n^{-(m+1)}\sum
_{j=1}^m\int_0^1
\mathbf {E}_{t_{i-1}} \bigl[\mathbf{1}_{\{
I_i+\rho
R_i>x-F_{i-1}\}}H^{t_{i-1}}_{(j,1,\ldots,m)}
\bigl(\mathcal{U}_{i}, R_i^j
\Phi_{M}\Phi _{c_i,\epsilon_i}\bigr) \bigr]\,d\rho.
\]
Hence the following bound holds true (see \cite{Nu06}, page 102):
\[
|J_{3,i}| \le c_{1,q} \sigma_n^{-(m+1)}
A_1 \int_0^1 A_2(
\rho) A_3(\rho) \,d\rho,
\]
where the quantities $A_1$, $A_2(\rho)$, $A_3(\rho)$ are, respectively,
defined by
\[
A_1=\max_{j=1,\ldots,m}\bigl\|R_i^j
\Phi_{M'}\bigr\|_{k_1,p_1,t_{i-1}},\qquad A_2(\rho)= \bigl\|\det\bigl(
\Gamma_{\mathcal{U}_{i},t_{i-1}}^{-1}\bigr)\Phi_{M'}\Phi
_{c_i,\epsilon_i}\bigr\| _{p_3,t_{i-1}}^{k_3}
\]
and
\[
A_3(\rho)= \|\mathcal{U}_{i} \Phi_{M'}
\|_{k_2,p_2,t_{i-1}}^{k_4},
\]
and where we also recall that $R^j_{i}$ is defined by (\ref{eq:def-Ri}).
Then the first inequality in~(\ref{eq:bnd-J3i}) follows from
Lemmas~\ref{lem:app1}, \ref{lem:A2} and \ref{lem:app2} which have been postponed
to the \hyperref[sec:app1]{Appendix},
and by choosing $\gamma$ such that $H-\frac{1}2<\gamma$. In order to go
from the first inequality in (\ref{eq:bnd-J3i}) to the second one, we
simply apply Lemma~\ref{lem:prop-partition}.
\end{pf}

\subsection{Lower bound}
Let us first summarize the considerations of the previous
section: starting from decomposition (\ref{eq:dcp-delta-Fi}) and
applying Corollary~\ref{cor:low-bnd-J1}, Propositions
\ref{prop:bound-J2}, \ref{prop:bound-norm-phi} and \ref{prop:bnd-J3i}
and the forthcoming relation
(\ref{rmk:behavior-constants}), we
have obtained the following facts: the inequality $\mathbf
{E}_{t_{i-1}} [\delta_x(F_i)] \ge J_{1,i} + J_{2,i} + J_{3,i}$ holds true, and thus
%
\begin{eqnarray}
\mathbf{E}_{t_{i-1}} \bigl[\delta_x(F_i)
\bigr] &\ge&  \frac{1}{(2\pi)^{m/2}(\Lambda\sigma_n^2)^{m/2}} \exp \biggl(-\frac{|x-F_{i-1}|^{2}}{2\lambda  \sigma_n^{2}} \biggr)
\nonumber
\\[-8pt]
\label{prop:summary-bnd-J}
\\[-8pt]
\nonumber
&&{}-c_{\lambda,\Lambda} \bigl({\sigma_n^2}
\bigr)^{-k_2} L_{n,i}^{\gamma,p}(k_1,p_1)
- \frac{c_{M,V,m}}{n^{\gamma/2H}  (\sigma_n^2)^{m/2} },
\end{eqnarray}
with the additional information $\mathbf{E}[ L_{n,i}^{\gamma
,p}(k_1,p_1) ]
\leq C_{M,\eta}
n^{-\eta}$ for an arbitrarily large exponent $\eta$.

We are now ready to prove the main theorem of this article:

\begin{pf*}{Proof of Theorem~\ref{thm:low-bnd-intro}}
With equation (\ref{prop:summary-bnd-J}) in hand, we shall
follow the strategy designed in~\cite{Ba,Ko}: Fix $x-a$ throughout the
proof, and define the
balls $B_i= B(y_i,  c_1\sigma_n)$ for $i=1, \ldots,  n$
where $y_i=a+\frac{i}{n}(x-a)$. We also define below an
additional sequence $\{x_i;  i=1,\ldots,n\}$, such that $x_i\in
B_i$ and $x_n=x$. The constant $c_1$ will be fixed later on (see Figure~\ref{fig1}).
\begin{figure}

\includegraphics{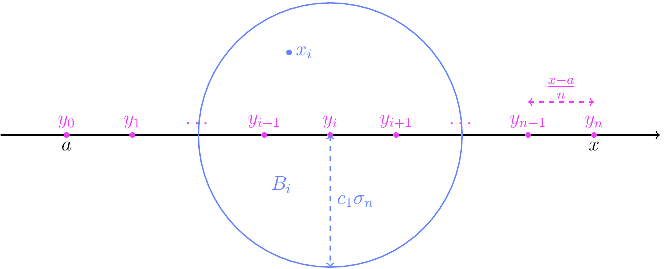}

\caption{Space partition for the lower bound, with sequence
$y_{1},\ldots,y_{n}$ and $x_{i}$.}\label{fig1}
\end{figure}

We shall now proceed in a backward recursive way on the index
$i$. For instance, in order to go from $n$ to $n-1$, we resort to (\ref{prop:summary-bnd-J}) in order to write
\[
\mathbf{E} \bigl[\delta_x(F_n) \bigr] = \mathbf{E}
\bigl[ \mathbf{E}_{t_{n-1}} \bigl[\delta_x(F_n)
\bigr]\bigr] \geq \frac{c_{V,m}}{\sigma_{n}^{m}} \mathbf{E} \biggl[\exp \biggl(-
\frac{|x-F_{n-1}|^2}{2\lambda\sigma
_n^2} \biggr) -c_{M,V,m} n^{-\kappa} \biggr],
\]
for a certain strictly positive constant $\kappa$. Hence
\begin{eqnarray*}
&&\mathbf{E} \bigl[\delta_x(F_n) \bigr]\\
 &&\qquad\ge
\frac{c_{V,m}}{\sigma_{n}^{m}} \int_{\mathbb{R}} \mathbf{E} \biggl[ \biggl(\exp
\biggl(-\frac{|x-F_{n-1}|^2}{2\lambda
\sigma_n^2} \biggr) -c_{M,V,m} n^{-\kappa} \biggr)
\delta_{x_{n-1}}(F_{n-1}) \biggr]\,dx_{n-1}
\\
&&\qquad\ge \frac{c_{V,m}}{\sigma_{n}^{m}} \int_{B_{n-1}} \mathbf{E} \biggl[
\biggl(\exp \biggl(-\frac{|x-F_{n-1}|^2}{2\lambda
\sigma_n^2} \biggr) -c_{M,V,m}
n^{-\kappa} \biggr)\delta_{x_{n-1}}(F_{n-1}) \biggr]
\,dx_{n-1}.
\end{eqnarray*}
We now observe the following: if we wish the term
$\delta_{x_{n-1}}(F_{n-1})$ to give a nonnull contribution,
the relations
\begin{eqnarray*}
x_{n-1}&\in&  B(y_{n-1}, c_{1}\sigma_{n}),\qquad
x-y_{n-1}=\frac{x-a}{n},
\\
\sigma_{n} &=& \frac{t^{H}}{n^{1/2}},\qquad
|F_{n-1}-x_{n-1}| \le c_{1} \sigma_n
\end{eqnarray*}
must be satisfied. Moreover,
from these conditions, it is easily seen that $|x-F_{n-1}|\le
4c_{1}\sigma_{n}$ whenever $n\ge\frac{|x-a|^2}{c_1t^{2H}}$.
We thus define a constant $c_2\ge\frac{1}{4c_1}$ such that
%
\begin{equation}
\label{eq:expression-n-optimd}
n=\frac{c_2   \llvert x-a\rrvert ^2}{t^{2H}}.
\end{equation}
Then if we take $c_1$ such that $\exp (
-\frac{8c_1^2}{\lambda} )\geq\frac{1}2$ and $n$ such that
$c_{M,V,m}   n^{-\kappa}\le1/4$, we obtain
\[
\mathbf{E} \bigl[\delta_x(F_n) \bigr] \geq
\frac{c_{V,m}}{4\sigma_{n}^{m}} \int_{B_{n-1}} \mathbf{E} \bigl[
\delta_{x_{n-1}}(F_{n-1}) \bigr]\,dx_{n-1}.
\]

These arguments can now be iterated backward from $i=n-1$ to
$1$, and the reader can easily check that the only additional
required condition is the compatibility relation
$y_{i+1}-y_{i}\leq c_1\sigma_n$ (this will be verified below).
Denoting by $\alpha_m$ the volume of a unit ball in $\mathbb{R}^{m}$
[viz. $\alpha_{m}=\pi^{m/2}/\Gamma(\frac{m}{2}+1)$], we end up
with
\begin{eqnarray}
\mathbf{E} \bigl[\delta_x(F_n) \bigr]
&\geq &  \biggl(\frac{c_{V,m}}{4 \sigma_{n}^{m}} \biggr)^n \bigl|B(0,c_1
\sigma_n)\bigr|^{n-1} \nonumber \\
&=&  \biggl(\frac{c_{V,m}}{4}
\biggr)^n \biggl(\frac{n^{1/2}}{t^{H}} \biggr)^{nm} \biggl(
\frac{c_1 t^H}{n^{1/2}} \biggr)^{m(n-1)} \alpha_{m}^{n-1}
\nonumber
\\[-8pt]
\label{eq:low-bnd-expp-form}
\\[-8pt]
\nonumber
&=& \biggl(\frac{c_{V,m}}{4} \biggr)^n \bigl(c_1^m
\alpha_{m}\bigr)^{n-1} \biggl(\frac{n^{1/2}}{t^{H}}
\biggr)^{m}\nonumber \\
&=& \frac{1}{\alpha_{m} (c_1t^{H} )^m} \exp \biggl(n\ln \biggl(
\frac{c_{V,m} c_1^m\alpha_{m}}{4} \biggr)+\frac{m}{2} \ln(n) \biggr).
\nonumber
\end{eqnarray}

Once here, we are reduced to tune our parameters according to
the following constraints:
\begin{longlist}[(ii)]
\item[(i)]
Recalling (\ref{eq:expression-n-optimd}), we have that if $c_1$ is
taken small enough so that
$\rho\equiv-\ln(c_{V,m} c_1^m\alpha_{m}/4)>0$ and (as alluded to
above) such that $\exp( -8c_1^2/\lambda)\geq\frac{1}2$ and $n\ln
(\rho
)+m\ln(n)\ge0$ for all $n\in\mathbb{N}$, we get
\[
\exp \biggl(n\ln \biggl(\frac{c_{V,m} c_1^m\alpha_{m}}{4} \biggr) \biggr) = \exp \biggl(-
\frac{\rho c_2\llVert x-a\rrVert ^2}{t^{2H}} \biggr).
\]
We remark here that the values of $c_1$, $c_2$ and $c_{M,V,m}$ are
fixed independently of $n$. It is
now easily seen that our bound (\ref{eq:low-bnd-expp-form}) is
of the form (\ref{eq:low-bnd-intro}).

\item[(ii)]
We now choose the constant $c_2$ in
(\ref{eq:expression-n-optimd}) so that the compatibility relation
$y_{i+1}-y_{i}\leq c_1\sigma_n$ is satisfied. Toward this aim,
recall that
\[
|y_{i+1}-y_{i}|=\frac{|x-a|}{n}=\frac{|x-a|}{n^{1/2}}
\frac{1}{n^{1/2}},
\]
and since $n= c_2\frac{|x-a|^2}{t^{2H}}$, we get
\[
|y_{i+1}-y_{i}|=\frac{|x-a|}{n^{1/2}}c_2^{-1/2}
\frac{t^H}{\llvert x-a\rrvert }=c_2^{-1/2} \sigma_n.
\]
\end{longlist}
It is thus sufficient to take $c_2^{-1/2}\leq c_1\wedge(2c_1^{1/2})$,
which also satisfies that $n\ge\frac{|x-a|^2}{4c_1t^{2H}}$. This
completes our proof.
\end{pf*}

\begin{appendix}
\section*{Appendix: Some properties of stochastic derivatives}
\label{sec:app1}

We start this technical section with a general bound on the space
$\mathcal{H}$ related to fBm.

\setcounter{theorem}{0}
\begin{lemma}\label{lem:bnd-H-C-gamma}
Let $H\in(0,1/2)$, $t\in(0,1]$ and consider the space $\mathcal{H}$
defined on
$[0,t]$ as in Section~\ref{sec:malliavin-tools}. Let $f$ be an element
of $\mathcal{C}^{\gamma}([0,t])$ for $1/2-H<\gamma<1/2$, with $\|f\|
_{\infty}\le a$
and $\|f\|_{0,t,\gamma}\le b$. Then
\[
\|f\|_{\mathcal{H}} \le c_{H} \bigl(a t^{H} + b
t^{\gamma+H} \bigr).
\]
\end{lemma}

\begin{pf}
For a function $g$ defined on $[0,t]$, recall that its fractional
derivative is given by
%
\begin{equation}
\label{eq:def-frac-deriv}
D_{t^{-}}^{1/2-H}g_{u} =
\frac{g_{u}}{(t-u)^{1/2-H}} + \int_{u}^{t}
\frac{g_{u}-g_{v}}{(v-u)^{3/2-H}} \,dv.
\end{equation}
Consider now $f\in\mathcal{C}^{\gamma}([0,t])$ satisfying the
conditions above,
and set $g_{u}=u^{-(1/2-H)}f_{u}$. According to \cite{Nu06}, formula (5.31), we have
%
\begin{equation}
\label{eq:bnd-H-frac-deriv} \|f\|_{\mathcal{H}}^{2}\le c_{H} \int
_{0}^{t} s^{1-2H} \bigl\llvert
D_{t^{-}}^{1/2-H}g_{s}\bigr\rrvert ^{2}
\,ds.
\end{equation}
We now proceed to estimate the right-hand side of relation (\ref{eq:bnd-H-frac-deriv}).

Indeed, plugging definition (\ref{eq:def-frac-deriv}) into (\ref
{eq:bnd-H-frac-deriv}), it is readily checked that
%
\begin{eqnarray}
\|f\|_{\mathcal{H}}^{2}&\le&  c_{H} \biggl(\int
_{0}^{t} A_{s}^{2} \,ds +
\int_{0}^{t} B_{s}^{2}
\,ds \biggr)\nonumber\\
\eqntext{\ds\mbox{with } A_{s} = \frac{f_{s}}{(t-s)^{1/2-H}},
B_{s} = \int_{s}^{t}
\frac{f_{s}- \psi_{v}   f_{v}}{(v-s)^{3/2-H}} \,dv,}
\end{eqnarray}
where we have set $\psi_{v}=(s/v)^{1/2-H}$. It is then easily seen that
$\int_{0}^{t} A_{s}^{2}   \,ds\le c_{H}   a^{2}   t^{2H}$. In order to
bound $B$, notice that the function $\psi$ is well defined on $[s,t]$
and satisfies $\psi_{s}=1$, $\psi_{v}\le1$ and $|\psi'_{v}|\le v^{-1}$.
\begin{eqnarray*}
\llvert f_{s}- \psi_{v} f_{v} \rrvert
&\le & |f_{s}- f_{v}| \psi_{v} +
|f_{s}| |1-\psi_{v}| \\
&\le& b (v-s)^{\gamma} + a |1-
\psi_{v}|^{\gamma} \le \biggl(b + \frac{a}{s^{\gamma}} \biggr)
(v-s)^{\gamma}.
\end{eqnarray*}
Dividing this inequality by $(v-s)^{3/2-H}$, recalling that $\gamma\le
1/2$ and integrating over $[s,t]$, we get
\[
|B_{s}| \le c_{H} \biggl(b + \frac{a}{s^{\gamma}} \biggr)
(t-s)^{\gamma-(1/2-H)},
\]
which entails that
\[
\int_{0}^{t} B_{s}^{2}
\,ds \le c_{H} \bigl(a^{2} t^{2H} +
b^{2} t^{2(\gamma+H)} \bigr).
\]
Gathering our bounds on $\int_{0}^{t} A_{s}^{2}   \,ds$ and $\int_{0}^{t} B_{s}^{2}   \,ds$, our proof is now complete.
\end{pf}

Let us now state a bound on Malliavin derivatives.

\begin{pf*}{Proof of relation (\ref{eq:DWDB})}
We focus on the first derivative case, the other ones being handled in
a similar fashion. We will thus prove that
\[
|\mathrm{D}_{u}F|\le\mathop{\ess\sup}_{u\le r}|
\mathbf{D}_{r}F|K(t,u).
\]
Indeed, according to Proposition~\ref{prop:relation-D-DW}, we have that
for $F\in\mathcal{F}_t$,
\[
|\mathrm{D}_{u}F|=\bigl|\bigl[K_t^*\mathbf{D}F
\bigr]_u\bigr|=\biggl|\int_u^t \mathbf
{D}_rF\,\partial_rK(r,u)\,dr\biggr| \le\mathop{\ess\sup}_{u\le r\le t}|\mathbf{D}_rF|K(t,u),
\]
which is exactly our claim.
\end{pf*}

We now turn to the bounds on the process $Q$ featuring in the
definition of our remainders $R_{i}$ [see decomposition (\ref
{eq:recurs-Fi-Fi-1}) of $X_{t}$]:

\begin{lemma} \label{lema:cotaQ}
Let $X$ be the solution to (\ref{eq:sde-Vd}), let $\eta_i$
be the function defined by~(\ref{eq:notation-eta-i-gk-i}) and $Q$ the
process given by (\ref{eq:def-Qk}). If
$r_1,s\in(t_{i-1},t_{i})$,
then the following bounds hold true:
%
\begin{eqnarray}\label{eq:Q-zero-order}
\bigl\llvert Q_s^k \bigr\rrvert &\leq&
c_V K(t,s) |t_i-t_{i-1}|^{\gamma}
Z^i_0,
\\
\label{eq:firstder}
\bigl\llvert \mathrm{D}_{r_1}^{l} Q_s^k
\bigr\rrvert &\leq& c_V K(t,s) K(t,r_1)
Z^i_1,
\end{eqnarray}
for $\mathcal{F}_1$-measurable random variables
$Z^i_0, Z^i_1$ defined by $Z^i_0=\|B\|_{t_{i-1},t,\gamma}\vee\|B\|
_{t_{i-1},t,\gamma}^\gamma$ and
%
\begin{equation}\label{eq:def-Z1}
Z^i_1=\sup \bigl\{\bigl|\mathbf{D}_{r_1}^{l}
(X_v-X_{t_{i-1}})\bigr|, t_{i-1}\leq r_1\leq
v\leq t_i \bigr\},
\end{equation}
admitting moments of all orders.
In general, we can extend these results to Malliavin derivatives
of arbitrary order $\ell\geq1$ in the following way: for
$r_1,s\in(t_{i-1},t_{i})$ and $r_2, \ldots, r_\ell<t_i$, we
have
%
\begin{equation}
\label{eq:gender}
\bigl|\mathrm{D}_{r_1,\ldots, r_\ell}^{j_1,\ldots, j_\ell} Q_s^k\bigr|
\leq c_V K(t,s) Z^i_\ell\prod
_{j=1}^n K(t, r_j),
\end{equation}
for $Z^i_\ell\equiv\sup \{|\mathbf{D}_{r_1,\ldots, r_\ell
}^{j_1,\ldots,
j_\ell} (X_v-X_{t_{i-1}})|, t_{i-1}\leq r_i\leq v\leq t_i,
i=1, \ldots, n \}$, which is a $\mathcal{F}_1$-measurable random
variable with moments of all orders.
\end{lemma}

\begin{pf}
Bound (\ref{eq:Q-zero-order}) is an easy consequence of
(\ref{eq:def-Qk}), Proposition~\ref{prop:cotes-hol-sol} and the fact
that $\partial
_uK(u,s)\ge0$. Moreover, observe that whenever $r_1>t_{i-1}$, we
have $\mathrm{D}_{r_1}V_k(X_{t_{i-1}})=0$. Hence, using
Proposition~\ref{prop:relation-D-DW}, we get
\begin{eqnarray*}
\bigl|\mathrm{D}_{r_1}^l Q_s^k\bigr|& =&
\biggl\llvert \int_{s\vee r_1}^t \partial_u
K(u,s) \mathrm{D}_{r_1}^l V_k(X_{\eta_i(u)})
\,du\biggr\rrvert 
\\
&=& \biggl\llvert \int_{s\vee r_1}^t
\partial_u K(u,s) \bigl[K^* _{t}\mathbf
{D}_{\cdot}^{l} V_k(X_{\eta_i(u)})
\bigr]_{r_1} \,du \biggr\rrvert
\\
&=&\biggl\llvert \int_{s\vee r_1}^t
\partial_u K(u,s) \biggl(\int_{r_1}^{t}
\mathbf{D}_{r_2}^{l} V_k(X_{\eta_i(u)})
\,\partial_{r_2} K(r_2,r_1) \,dr_2
\biggr)\,du \biggr\rrvert.
\end{eqnarray*}
It is thus readily checked that
\[
\bigl|\mathrm{D}_{r_1}^l Q_s^k\bigr| \leq
c_V Z^i_{1} \biggl\llvert \int
_{s\vee r_1}^t \partial_u K(u,s)
K(t,r_1) \,du \biggr\rrvert \leq c_V
Z^i_{1} K(t,s) K(t,r_1).
\]
%
The general result (\ref{eq:gender}) is now obtained by means of
an induction argument and resorting to the same techniques as in
the case of the first order derivative (namely $\ell=1$).
\end{pf}

\begin{remark}\label{rmk:bnd-Z1-Zell}
Note that due to the definition (\ref{eq:def-Z1}) of $Z^i_l$ and
Proposition~\ref{prop:cotes-hol-sol} which controls the derivatives of
$X$ using the H\"older norms of $B$, the random
variables $Z$ verify
\[
\bigl|Z_{j}^{i}\bigr|\leq C_V\exp
\bigl(C_V \|B\|_{t_{i-1},t_i,\gamma}^{{1}/{\gamma}} \bigr),
\]
for any $\gamma\in(\frac{1}2,H)$. Hence, applying Proposition~\ref{prop:rel-normB-N} we obtain
\[
\bigl|Z_{j}^{i}\bigr|\leq C_V\exp
\bigl(C_{V,\gamma} \bigl( \mathrm{N}_{\gamma,p}^i(B)
\bigr)^{{1}/{2\gamma p}} \bigr),
\]
for any $p$ such that $0<\gamma<H-\frac{1}{2p}$. This relation
yields in particular that $Z_{j}^{i}\in\bigcap_{q\ge
1}L^{q}(\Omega)$. Furthermore, once we localize by the random variables
$\Phi_{M}$ or $\Phi_{M'}$, we end up with
%
\begin{equation}
\label{rmk:behavior-constants}
\quad\max_{0\le l\le k} \bigl(Z^i_l
\Phi_{M'} \bigr)\le c_{M,V,m}\qquad \mbox{with } c_{M,V,m} =
c_{V,m}\exp\bigl(c_{V,m}\bigl(M'
\bigr)^{{1}/{2\gamma p}}\bigr).
\end{equation}
\end{remark}

In the next proposition, we give norm estimates for the remainder terms
$R_i$ needed in the upper bound for $J_{3,i}$.

\begin{lemma}
\label{lem:app1}
In the setting of Proposition~\ref{prop:Xt-strato-B} and Corollary~\ref{cor:eq-X-strato-wrt-W}, with definition
(\ref{eq:def-Ri}) and (\ref{eq:star}),
the following estimate is valid:
%
\begin{equation}
\label{eq:bnd-Rj}
\bigl\|R^j_i\Phi_{M'}
\bigr\|_{k_1,p_1,t_{i-1}} \le c_{V,M} (t_{i}-t_{i-1})^{\gamma
}\sigma_{n}.
\end{equation}
\end{lemma}

\begin{pf}
This result obviously involves the control of many derivative terms.
For the sake of conciseness, we only sketch the bound for $\mathrm
{D}R_{i}$. Now recall that
\[
R_{i}= \sum_{k=1}^d\int
_{t_{i-1}}^{t_{i}} Q^k_s \circ
dW_{s}^k.
\]
We now apply a small variant of \cite{Nu06}, Proposition~1.3.8, to
Stratonovich integrals, which states that for $r\in[t_{i-1},t_{i}]$,
we have
%
\begin{equation}
\label{eq:deriv-R} \mathrm{D}^{j}_{r}R_{i}=
Q_{r}^{j} + \sum_{k=1}^d
\int_{t_{i-1}}^{t_{i}} \mathrm{D}^{j}_{r}Q^k_s
\circ dW_{s}^k.
\end{equation}
Let us now evaluate the $L^{2}[t_{i-1},t_i]$ norm of $\mathrm
{D}^{j}_{r}R_{i}$. The main contribution for this norm comes from the
term $Q$ on the right-hand side of (\ref{eq:deriv-R}), for which we
obtain, according to (\ref{eq:Q-zero-order}),
\begin{eqnarray*}
\int_{t_{i-1}}^{t_{i}} \bigl(Q_{r}^{j}
\bigr)^{2} \,dr&\le& c_{V} |t_{i}-t_{i-1}|^{2\gamma}
\bigl(Z_{0}^{i} \bigr)^{2} \int
_{t_{i-1}}^{t_{i}} K^{2}(t,r) \,dr
\\
&=& c_{V} \bigl(Z_{0}^{i}
\bigr)^{2} |t_{i}-t_{i-1}|^{2\gamma
}
\sigma_{n}^{2},
\end{eqnarray*}
and thus
\begin{eqnarray*}
\mathbf{E}_{t_{i-1}}^{1/p_{1}} \bigl[\|Q\| _{L^{2}([t_{i-1},t_{i}])}^{p_{1}}
\Phi _{M'} \bigr] &\le&  c_{V} |t_{i}-t_{i-1}|^{\gamma}
\sigma_{n} \mathbf{E}_{t_{i-1}}^{  p_{1}} \bigl[
\bigl(Z_{0}^{i} \bigr)^{p_{1}}
\Phi_{M'} \bigr]\\
& \le &  c_{V,M} |t_{i}-t_{i-1}|^{\gamma}
\sigma_{n},
\end{eqnarray*}
which is consistent with our claim (\ref{eq:bnd-Rj}).

Let us give another example of term which has to be analyzed in order
to bound the norm of $\mathrm{D}^{j}_{r}R_{i}$: the term $A$ defined as
\[
A := \mathbf{E}_{t_{i-1}}^{1/p_{1}} \biggl[ \biggl(\int
_{t_{i-1}}^{t_{i}} dr \int_{t_{i-1}}^{t_{i}}
ds \bigl[\mathrm {D}^{j}_{r}Q^k_s
\bigr]^{2} \biggr)^{{p_{1}}/{2}} \Phi_{M'} \biggr].
\]
Along the same lines as above, using (\ref{eq:Q-zero-order}), we find
\[
A \le c_{M,V} \int_{t_{i-1}}^{t_{i}} ds
K^{2}(t,s) \int_{t_{i-1}}^{t_{i}} dr
K^{2}(t,r) = c_{M,V} \sigma_{n}^{4},
\]
which is a remainder term with respect to (\ref{eq:bnd-Rj}). Notice
that many other higher order terms have to be evaluated in order to
complete the proof. We omit these cumbersome but routine developments
for sake of conciseness.
\end{pf}

We now turn to the bound on $A_2(\rho)$:

\begin{lemma} \label{lem:A2}
Recall that $A_2(\rho)$ is defined as $A_2(\rho)=
\|\det(\Gamma_{\mathcal{U}_{i},t_{i-1}})^{-1}\Phi_{M'}\times\Phi
_{c_i,\epsilon_i}\|
_{p_3,t_{i-1}}^{k_3}$. Then this quantity is uniformly bounded in $n$,
$\rho$ and $\omega\in\Omega$.
\end{lemma}

\begin{pf}
Recall that $\mathcal{U}_{i}=\sigma_{n}^{-1} (I_{i}+\rho R_i)$, and
remark that
using Proposition~4 in \cite{Ba}, we have that
\[
\det(\Gamma_{\mathcal{U}_{i},t_{i-1}})^{-1}\Phi_{c_i,\epsilon_i} \le
\sigma_{n}^{2m} \Biggl(\frac{1}2\lambda\int
_{t_{i-1}}^{t_i}K^2(t,s)\,ds -\sum
_{j=1}^d\int_{t_{i-1}}^{t_i}\bigl|
\mathrm{D}^j_rR_i\bigr|^2\,dr
\Biggr)^{-m}\Phi_{c_i,\epsilon_i}.
\]
Moreover, we have localized $\sum_{j=1}^d\int_{t_{i-1}}^{t_i}|\mathrm
{D}^j_rR_i|^2\,dr$ by $\Phi_{c_i,\epsilon_i}$ with $c_{i}=\frac
{\lambda\sigma
_{n}^{2}}{8}$. Thus we end up with
\[
\det(\Gamma_{\mathcal{U}_{i},t_{i-1}})^{-1}\Phi_{c_i,\epsilon_i} \le
\sigma_{n}^{2} \biggl(\frac{\lambda}{4} \int
_{t_{i-1}}^{t_i}K^2(t,s)\,ds
\biggr)^{-1},
\]
from which the result follows.
\end{pf}

The estimates for $A_3(\rho)$ are obtained in a similar fashion.
In fact, we have:

\begin{lemma}
\label{lem:app2}
The same conclusion as in Lemma~\ref{lem:A2} holds true for the
quantity $A_3(\rho)=
\|\mathcal{U}_{i}   \Phi_{M'}\|_{k_2,p_2,t_{i-1}}^{k_4}$.
\end{lemma}

\begin{pf}
With respect to Lemma~\ref{lem:app1}, we only need to consider
additionally the bound for
\[
\|I_i\Phi_{M'}\|_{k_2,p_2,t_{i-1}}\le c\|I_i
\|_{k_2,p_3,t_{i-1}}\|\Phi_{M'}\|_{k_2,p_4,t_{i-1}}.
\]
The above follows from H\"older's inequality. Therefore the result
follows from straightforward calculations for $I_i$ as in the proof of
Proposition~\ref{prop:bound-J2}.
\end{pf}
\end{appendix}

\section*{Acknowledgment}

M. Besal\'{u} and S. Tindel are members of the BIGS
(Biology, Genetics and Statistics) team at INRIA.


%

%



\printaddresses
\end{document}